\documentclass{article}

\usepackage{graphicx} 
\usepackage[utf8]{inputenc} 
\usepackage[T1]{fontenc}    
\usepackage[english]{babel} 
\usepackage{amsmath, amssymb, amsthm} 
\usepackage[a4paper,top=3cm,bottom=3cm,left=2.1cm,right=2.1cm,marginparwidth=1.75cm]{geometry}
\usepackage{parskip} 
\usepackage[colorlinks=true, allcolors=blue]{hyperref}    
\usepackage{tocloft}        
\usepackage{braket}
\usepackage{cleveref}
\usepackage{tikz}
\usepackage{multirow}
\usepackage{booktabs}
\usepackage{tabularx}
\usepackage{nicematrix}
\usepackage{bbm}
\usepackage{enumitem}
\usepackage{algorithm}
\usepackage{algpseudocode}
\usepackage[normalem]{ulem}

\usepackage[round]{natbib}
\usepackage[textsize=small]{todonotes}

\newcommand{\R}{\mathbb{R}}
\newcommand{\N}{\mathbb{N}}
\newcommand{\1}{\mathbbm{1}}

\usepackage{authblk}

\newtheorem{theorem}{Theorem}[section]
\newtheorem{lemma}[theorem]{Lemma}

\newtheorem{proposition}[theorem]{Proposition}

\newtheorem{remark}[theorem]{Remark}

\numberwithin{equation}{section}
\title{Efficient Simulation of Hawkes Processes using their Affine Volterra Structure
}
\author[1]{Eduardo Abi Jaber\thanks{EAJ is grateful for the financial support from the Chaires FiME-FDD and Financial Risks at Ecole Polytechnique.}}

\author[1]{Elie Attal} 

\author[1, 2]{Dimitri Sotnikov\thanks{DS is grateful for the financial support provided by Engie Global Markets.}}

\affil[1]{Ecole Polytechnique, CMAP}
\affil[2]{Engie Global Markets}

\begin{document}

\maketitle

\begin{abstract}
    We introduce a novel and efficient simulation scheme for Hawkes processes on a fixed time grid, leveraging their affine Volterra structure. The key idea is to first simulate the integrated intensity and the counting process using Inverse Gaussian and Poisson distributions, from which the jump times can then be easily recovered. Unlike conventional exact algorithms based on sampling jump times first, which have random computational complexity and can be prohibitive in the presence of high activity or singular kernels, our scheme has deterministic complexity which enables efficient large-scale Monte Carlo simulations and facilitates vectorization. Our method applies to any nonnegative, locally integrable kernel, including singular and non-monotone ones. By reformulating the scheme as a stochastic Volterra equation with a measure-valued kernel, we establish weak convergence to the target Hawkes process in the Skorokhod $J_1$-topology. Numerical experiments confirm substantial computational gains while preserving high accuracy across a wide range of kernels, with remarkably improved performance for a variant of our scheme based on the resolvent of the kernel.
\end{abstract}

\section{Introduction}
Hawkes processes, originally introduced by \citet{hawkes1971spectra}, form a class of \textit{self-exciting} counting processes where each event increases the likelihood of future events. This self-excitation is governed by a convolution kernel $K \in L^1([0\,,T]\,, \R_+)\,$. More precisely, the  counting process $N = \left(N_t\right)_{t \leq T}$ is characterized by its stochastic instantaneous intensity
\begin{equation}
\label{eq:intensity}
\lambda_t = g_0(t) + \int_0^t\,K(t-s)\,dN_s\,,\quad t \leq T\,,
\end{equation}
where $g_0 \in L^1([0\,,T]\,, \mathbb R_+)$ is the \textit{exogenous intensity}, driving events that occur independently of the past trajectory. The second term, $K * dN\,$, is the \textit{endogenous} component, which increases the intensity at each jump of $N$ and retains memory of these jumps through the kernel $K\,$. We refer to \citet*{laub2025hawkes} for a general review on Hawkes processes.

When $K(t) = c\,e^{-b\,t}\,$, with $c > 0\,, b \in \R\,$, the process $(N, \lambda)$ is Markovian. However, for general kernels $K\,$, the Volterra-type memory breaks the Markov property. This ability to encode self-excitation with memory makes Hawkes processes popular in diverse domains, including financial mathematics \citep*{bacry2015hawkes, hawkes2018hawkes, embrechts2011multivariate, rosenbaum2021microscopic}, insurance \citep*{lesage2022hawkes, bessy2021multivariate, baldwin2017contagion, hillairet2023expansion}, seismology \citep*{ogata1988statistical, davis2024fractional} and modeling of biological neurons \citep{hodara2017hawkes, goncalves2022perfect}.

In the non-Markovian setting, computing key quantities, such as the distribution of $N_t$, its moments, or its quantiles, is generally intractable in closed form. Consequently, simulation  for Monte Carlo methods plays a central role in applications.

One of the most widely used simulation methods is \textit{Ogata's thinning algorithm} from \citet{Ogata81}, based on the acceptance-rejection technique, which generates candidate jumps and discards some of them at random. It requires explicit upper bounds on the intensity $\lambda_t$, which are typically available when $K$ is bounded. However, the scheme is no longer exact when $K$ is unbounded (e.g., $K(0) = +\infty$), and the case of a non-monotone kernel requires special treatment.

Another possibility for simulation is to use the \textit{population} representation (also called immigrant-birth or cluster representation, see \citet{hawkes1974cluster, Moller2005}), in which the Hawkes process is seen as a branching structure: each jump gives birth to an inhomogeneous Poisson process, which can be simulated directly. This approach only requires $K$ to be integrable.

Both simulation methods share the same important drawback: their computation cost is random, since jump times are simulated first, before building the processes $\lambda$ and $N\,$. Therefore, their complexity has order of magnitude $\mathcal O(N_T^2)\,$, where the square comes from the non-Markovian Volterra structure. In regimes with strong self-excitation, $N_T$ can be very large, making even a single simulation costly. Furthermore, this makes vectorized sampling more difficult in the context of Monte Carlo methods. In addition, in many applications, the quantities of interest depend on the finite-dimensional distributions of 
$N$ (see, for instance, \citet*{lesage2022hawkes, bessy2021multivariate}), rather than on the event times themselves. In such cases, simulating the individual jump times, with their inherently random computational complexity, becomes undesirable and may substantially degrade performance. Even if one restricts to the Markovian case, where 
$K$ is an exponential kernel, an exact simulation method exists that achieves linear, but still stochastic,  complexity of order 
 $\mathcal O (N_T)$ (see \citet{Dassios13}). 
 
 This naturally leads to a central question:

\begin{center}
\textit{Can Hawkes processes be simulated efficiently with a deterministic complexity?}
\end{center}

To address this question, we make two key observations: 
\begin{enumerate}
    \item \textbf{Only integrated quantities matter:} Once the integrated intensity
$\Lambda_{t_i,t_{i+1}} := \int_{t_i}^{t_{i+1}} \lambda_s ds$
has been simulated over a deterministic interval \([t_i, t_{i+1})\), the increment of the Hawkes process 
$N_{t_{i+1}} - N_{t_i}$, conditional on $\Lambda_{t_i,t_{i+1}}$, can be sampled using a  Poisson distribution with parameter $\Lambda_{t_i,t_{i+1}}$.  
\item  \textbf{Hawkes processes are Affine Volterra processes with jumps:}  as shown in \citet[Example~2.3]{abi2021weak}, \citet*{pulido2024affine, cuchiero2019markovian}. More precisely, based on the first observation, we focus on the dynamics of the (less conventional) integrated intensity \(\Lambda\), rather than the instantaneous intensity \(\lambda\). Integrating \eqref{eq:intensity} and applying Fubini's theorem yields
\begin{equation}
\label{eq:integrated_intensity}
\Lambda_t := \int_0^t \lambda_s \, ds 
= \int_0^t g_0(s)\, ds + \int_0^t K(t-s) (\Lambda_s + Z_s) \, ds, \quad t \le T,
\end{equation}
where $ Z := N - \Lambda$ is a martingale with predictable quadratic variation equal to $\Lambda$, which, in other words, is \textit{affine} in $\Lambda$.  This affine Volterra structure translates into the knowledge of the characteristic function of the integrated intensity $\Lambda$ in semi-explicit form in terms of a non-linear deterministic Volterra equation. 
\end{enumerate}

Recently, in the context of continuous affine Volterra processes (see, e.g., \citet*{abi2019affine}), the affine structure of integrated quantities has been leveraged to develop an efficient simulation method, namely the iVi scheme of \citet{jaber2024simulation, jaber2025simulating}, which relies on Inverse Gaussian increments.

Building on these ideas, we introduce a new paradigm for approximating $(\Lambda, N)$ directly on a fixed time grid $\left(t^n_i\right)_{0 \leq i \leq n} := \left(i\frac{T}{n}\right)_{0 \leq i \leq n}$ with $n+1$ points. Our method replaces the random $\mathcal O(N_T^2)$ complexity of classical schemes with a deterministic $\mathcal O(n^2)\,$, giving explicit control over the trade-off between computational cost and accuracy in Monte Carlo experiments. The scheme simulates the increments $\widehat \Lambda_{i,i+1}^n \approx \Lambda_{t^n_{i+1}} - \Lambda_{t^n_i}$ through Inverse Gaussian sampling, where the parameters are derived from an approximation on the characteristic function of $\Lambda\,$, followed by Poisson sampling for $\widehat N^n_{i,i+1} \approx N_{t^n_{i+1}} - N_{t^n_i}\,$. The approximation of the integrated intensity is then updated to reflect the dynamic \eqref{eq:integrated_intensity}. As an optional step,  the jump times 
$\mathcal{T}^n = (\tau^n_k)_{k=1,\ldots,N_T^n}$ 
of the counting process ${N}$, where $N^n_T= \sum_{i=0}^{n-1} \widehat N_{i,i+1}^n$ denotes the total number of jumps on $[0,T]$,  can be {approximated} from $\widehat{N}^n$. 
On each interval $[t_i, t_{i+1})$, conditional on the number of jumps $\widehat{N}^n_{i,i+1}$ in that interval, we sample  $\widehat{N}^n_{i,i+1}$ jump times uniformly within the interval and then sort them in increasing order. The scheme is detailed in Algorithm~\ref{algo:ivi_hawkes}. We refer to our method as {Hawkes iVi}, for {integrated Volterra implicit}  following the terminology of \citet{jaber2024simulation, jaber2025simulating}.


\begin{algorithm}[H]\label{algo:ivi_hawkes}
\caption{ \textbf{- The Hawkes iVi scheme:} Simulation of \( \widehat \Lambda\,, \widehat N \)}\label{alg:simulation}
\begin{algorithmic}[1]
\State \textbf{Input:} $g_0\,,K \in L^1([0\,,T]\,, \R_+)$, uniform partition of $[0,T]$, \( t_i^n = i T/n  \), with $i=0,\ldots, n$.
\State \textbf{Output:} \(  \widehat \Lambda^n_{i, i+1}\,, \widehat N^n_{i, i+1} \) for \( i = 0, \ldots, n-1 \) and (optional) $\mathcal{T}^n = (\tau_k^n)_{k=1, \ldots, N_T^n}$.
\State Compute
\begin{align}\label{eq:kij}
k_{j}^n = \int_{t^n_j}^{t^n_{j+1}}\, K\left( s\right) \, ds\,, \quad {j=0\,,\ldots\,, n-1}.
\end{align}
\For{$i = 0$ to $n-1$}

\State Compute the quantity:
\begin{align}\label{eq:alphai}
    \alpha^n_i  = \int_{t_i^n}^{t^n_{i+1}} g_0(s) ds + \sum_{j=0}^{i-1}   k^n_{i-j}\,\widehat N^n_{j,j+1}\,,  
\end{align}
with  $\alpha^n_0=\int_0^{t_1^n} g_0(s) ds,$  for $i=0$\,.
    \State Simulate the increment:
\begin{align}
    \label{eq:hatNsample}
    \widehat N^n_{i,i+1} \sim \mathcal P\left(\xi^n_i\right)\,, \quad \text{where} \quad \xi^n_i \sim IG\left(\frac{\alpha^n_i}{1 - k_0^n}\,, \left(\frac{\alpha^n_i}{k_0^n} \right)^2 \right)\,.
\end{align}
    \State Compute the increment:
\begin{equation}
    \label{eq:hatLambdasample}
    \widehat \Lambda^n_{i,i+1} = \alpha^n_i + k_0^n\, \widehat N^n_{i,i+1}\,.
\end{equation}
\State {Optional step (jump times):} Simulate $\widehat N^n_{i,i+1}$ i.i.d. uniform random variables on $[t_i, t_{i+1})$, sort, and append them to $\mathcal{T}^n$.

\EndFor
\end{algorithmic}
\end{algorithm}
$IG(\cdot,\cdot)$ and $\mathcal P(\cdot)$ denote respectively the Inverse Gaussian (see Appendix \ref{App:IG}) and the Poisson distributions. \\

The Inverse Gaussian sampling in \eqref{eq:hatNsample} is justified by an approximation of the characteristic functional of $\Lambda_{t^n_{i+1}} - \Lambda_{t^n_i}\,$. Conditional on the latter, $N_{t^n_{i+1}} - N_{t^n_i}$ follows a Poisson distribution with parameter $\Lambda_{t^n_{i+1}} - \Lambda_{t^n_i}\,$. Since the scheme outputs approximations of the increments of $\Lambda$ and $N$ on the time grid, the piecewise constant càdlàg processes we are interested in are given by 
\begin{equation}
    \label{eq:def_process_Lambda}
\Lambda^n_t := \sum_{i = 0}^{\lfloor n t/T\rfloor - 1}\,\widehat \Lambda^n_{i,i+1}\,, \quad t \leq T\,, \quad n \geq 1 \,,
\end{equation}
\begin{equation}
    \label{eq:def_process_N}
    N^n_t := \sum_{i = 0}^{\lfloor nt/T\rfloor - 1}\,\widehat N^n_{i,i+1}\,, \quad t \leq T\,, \quad n \geq 1\,.
\end{equation}
Our main theoretical contribution, stated in Theorem~\ref{theorem:convergence}, establishes the weak convergence of the pair of processes $(\Lambda^n, N^n)$, defined in \eqref{eq:def_process_Lambda}--\eqref{eq:def_process_N}, to the limiting pair $(\Lambda, N)$ in the Skorokhod $J_1$-topology. The general strategy follows the classical approach for weak convergence: proving tightness and identifying the limit. However, each of these steps is far from straightforward in our setting and requires a novel analysis.  The key idea lies in recasting the discrete scheme as a stochastic Volterra equation with a measure-valued kernel (Lemma~\ref{lemma:Volterra}). This reformulation is crucial, as it reveals that the coupling between $N^n$ and $\Lambda^n$ naturally inherits a Volterra-type structure, and it enables refined moment  estimates, leading to tightness of the approximating sequence and a clear identification of the limiting dynamics through stability results for stochastic Volterra equations (Lemma~\ref{lemma:stability_equation}). Furthermore, a delicate point in our proof  is to show that the limit process $N$ retains the point-process structure with unit jumps, a nontrivial property that is not automatic under weak convergence in $J_1$.

Beyond eliminating random complexity, our approach offers several other benefits:
\begin{itemize}
    \item[$\bullet$] \textbf{Vectorization-friendly:} choosing $n$ before simulation enables efficient, parallel Monte Carlo sampling.
    \item[$\bullet$] \textbf{Robust kernel framework:} the method uses integrated kernel quantities only, accommodating any $K \in L^1([0\,,T], \R_+)\,$, including non-monotone and singular kernels.
    \item[$\bullet$] \textbf{Weak convergence guarantee:} we prove that the couple of processes $\left(\Lambda^n\,,N^n\right)$ defined in \eqref{eq:def_process_Lambda} and \eqref{eq:def_process_N} converges weakly to $(\Lambda\,, N)$ in the Skorokhod topology, by reformulating the scheme as a stochastic Volterra equation with a measure-valued kernel (see Section \ref{sect:convergence}).
    \item[$\bullet$] \textbf{Possible simulation of jump times:} although Algorithm \ref{algo:ivi_hawkes} simulates $(\Lambda\,,N)$ directly, which is more convenient for Monte Carlo estimation, jump times can still be approximated satisfactorily (see Subsection~\ref{subsec:numerical_experiments}).
      \item[$\bullet$] \textbf{Linear complexity in the Markovian case:} for exponential kernels $K(t) = c\,e^{-b t}$ with $c > 0\,,b \in \R\,$, a linear $\mathcal O(n)$ variant exists. To see this, we remark that, in this case, \eqref{eq:alphai} can be rewritten as
    \begin{equation}
        \label{eq:Lambda_markov_case}
        \alpha^n_i = \int_{t^n_i}^{t^n_{i+1}}\,g_0(s)\,ds +  e^{-b\,T/n}\,\left(\alpha^n_{i-1} - \int_{t^n_{i-1}}^{t^n_i}\,g_0(s)\,ds \right) + k^n_1 \,\widehat N^n_{i-1,i} \,,\quad 1 \leq i \leq n-1\,,
    \end{equation}
    which can be computed in $\mathcal O(1)$ time. This particular case is developed in Subsection~\ref{subsec:markov_scheme}. 
    \item[$\bullet$] \textbf{Reformulation using resolvent kernels:} replacing $K$ by its resolvent of the second kind drastically improves performance empirically, although positivity of the scheme is no longer guaranteed. This alternative scheme is tested in Subsection~\ref{subsec:resolvent_scheme}.
\end{itemize}

All in all, our numerical experiments in  Section \ref{sect:numerics} confirm substantial computational gains while preserving high accuracy across a wide range of kernels, with remarkably improved performance for the variant of our scheme based on the resolvent of the kernel.

\paragraph{Connection to the literature.} 
This work is inspired by the iVi scheme introduced by \citet{jaber2024simulation, jaber2025simulating} for continuous affine Volterra processes, and further enhanced by \citet{zaugg2025liftedhestonmodelefficient} for the lifted Heston model. Extending the approach to point processes requires several important adjustments. First, although both schemes rely on an Inverse Gaussian approximation, the sampling is performed differently in our jump setting to preserve the Volterra structure of the process $\Lambda$ while ensuring that $N$ remains integer-valued, a crucial distinction compared to the continuous case. Second, the proof of convergence necessitates new estimates, particularly to establish that the limiting process is a point process with unit jumps. From a numerical perspective, the iVi Hawkes scheme differs substantially from the simulation methods commonly used in the literature due to its deterministic complexity, and therefore requires a dedicated comparison to assess its computational advantages. Overall, while both iVi schemes share the same fundamental structure, each introduces improvements in its respective domain, demonstrating the effectiveness of the method for approximating affine Volterra processes in different contexts.

In the literature, weak convergence results towards Hawkes-type limits are rather scarce. To the best of our knowledge, \cite{kirchner2016hawkes,kirchner2017estimation} is one of the very few works studying such a convergence result using discrete-time  INAR($p$) and INAR($\infty$) processes. One key difference with our approach is that we also establish the convergence of the intensity process $\Lambda^n$, whereas \cite{kirchner2016hawkes} focuses only on the point process itself. A numerical comparison of Algorithm \ref{algo:ivi_hawkes} with a refined version of the approximation method of \citet{kirchner2016hawkes, kirchner2017estimation} is provided in Subsection \ref{subsec:numerical_experiments}.  More general weak convergence and stability results for stochastic Volterra equations with jumps, such as those in \cite*{abi2021weak, abietal2021weak}, cannot be directly applied in our context. These works rely on kernels $t \mapsto K(t)$ belonging to some $L^p$ space, while our setting is inherently discrete. After reformulation, our scheme can indeed be seen as a stochastic Volterra equation, but with a measure-valued kernel $K(ds)$, which requires a distinct treatment.

The most renowned simulation schemes for the Hawkes process, to which we compare our method in Section \ref{sect:numerics}, are Ogata's algorithm \citep{Ogata81}, the  population approach \citep{Moller2005} and the exact exponential scheme for the Markovian case \citep{Dassios13}. Other methodologies can be found in the literature, often adapted to a specific class of kernels. Among these, \citet*{Chen2021} propose to adapt Ogata's algorithm by adding a small time-shift to handle singular Mittag--Leffler kernels, while \citet*{Duarte} introduce a thinning algorithm for sums of Erlang kernels. 

\paragraph{Outline of the paper.} In Section \ref{sect:scheme_motivation}, we explain how the scheme can be derived from a joint approximation of the dynamics and of the characteristic functional of the integrated intensity, while Section \ref{sect:convergence} focuses on its weak convergence. Section \ref{sect:numerics} contains all numerical tests, and Section \ref{sect:convergence_proof} is dedicated to proving the weak convergence result of Theorem \ref{theorem:convergence}. Finally, Appendix \ref{app:kernels} defines the resolvent of kernel and gives their explicit form for some usual kernels, and Appendix \ref{App:IG} contains properties of the Inverse Gaussian distribution.

{\paragraph{Notation.} Throughout the paper, we use the symbol $*$ to denote convolution: 
$$
 (R*K)(t) := \int_0^tR(t-s)K(s)\,ds.
$$
}

\section{Deriving the Hawkes iVi scheme}\label{sect:scheme_motivation}
The scheme is based on a joint approximation on the dynamics \eqref{eq:integrated_intensity} and on the characteristic functional of the integrated intensity $\Lambda_t := \int_0^t\,\lambda_s\,ds\,$. We insist on the fact that the present paper deals with Hawkes process on $[0\,,T]\,$, as opposed to the whole nonnegative real line. Thus, we only require that $g_0\,,K \in L^1([0\,,T]\,,\R_+)\,$. The processes $\left(\Lambda_t\,,N_t\right)_{t \leq T}$ can be interpreted as the restriction to $[0\,,T]$ of the Hawkes process on $\R_+\,$, with compactly supported exogenous intensity and memory kernel respectively given by 
$$
\tilde g_0(t) := g_0(t)\,\1_{t \leq T}\,, \quad \text{and}\quad \tilde K(t) := K(t)\,\1_{t \leq T}\,, \quad t \geq 0\,.
$$
Note, however, that Algorithm \ref{algo:ivi_hawkes} is well-defined as long as $k_0^n := \int_0^{T/n}\,K(s)\,ds < 1\,$, due to \eqref{eq:hatNsample}, which is always true for $n$ large enough.

\subsection{Approximating the dynamics}
\label{subsec:dynamic_approx}
We recall that integrating \eqref{eq:intensity} and applying Fubini's interchange leads to the dynamics of the integrated intensity
$$
\Lambda_t = \int_0^t\,g_0(s)\,ds + \int_0^t\,K(t-s)\,N_s\,ds\,,\quad t \leq T\,.
$$
Thus, similarly to \citet[Proposition~1.1]{jaber2025simulating}, we can decompose the increments $\Lambda_t - \Lambda_s$ into an $\mathcal F_s$-measurable part, corresponding to the non-Markoviannity of the dynamics, and an increment between $s$ and $t\,$, as follows
$$
\Lambda_t - \Lambda_s = G_s(t) + \int_s^t\,K(t-r)\,(N_t - N_r)\,dr\,, \quad s \leq t \leq T\,, 
$$
where $G_s(t)$ is $\mathcal F_s$-measurable and given by 
\begin{equation}
\label{eq:G_s_t}
G_s(t) := \int_s^t\,g_0(r)\,dr + \int_0^s\,\int_s^t\,K(u-r)\,du\,dN_r\,,\quad s \leq t \leq T\,.
\end{equation}
We now want to approximate $\Lambda_{t^n_{i+1}} - \Lambda_{t^n_i}\,$. To do so, we derive a discretized proxy of $G_{t^n_i}(t^n_{i+1})$ using left-endpoint approximations of the integrals:
\begin{align*}
    G_{t^n_i}(t^n_{i+1}) &= \int_{t^n_i}^{t^n_{i+1}}\,g_0(s)\,ds + \sum_{j= 0}^{i-1}\,\int_{t^n_j}^{t^n_{j+1}}\,\int_{t^n_i}^{t^n_{i+1}}\,K(u-r)\,du\,dN_r \\
    &\approx \int_{t^n_i}^{t^n_{i+1}}\,g_0(s)\,ds + \sum_{j= 0}^{i-1}\,\int_{t^n_j}^{t^n_{j+1}}\,\int_{t^n_i}^{t^n_{i+1}}\,K(u-t^n_{j})\,du\,dN_r \\
    &\approx \int_{t^n_i}^{t^n_{i+1}}\,g_0(s)\,ds + \sum_{j= 0}^{i-1}\,k^n_{i-j}\,\widehat N_{{j, j+1}}^n  =: \alpha_i^n\,, \quad i \leq n-1\,, \quad n \geq 1\,,
\end{align*}
where the $\left(k^n_i\right)_{i \leq n-1}$ are defined in \eqref{eq:kij}. Hence, we have done the approximation $G_{t^n_i}(t^n_{i+1}) \approx \alpha^n_i\,$, which explains the definition of $\left(\alpha^n_i\right)_{i \leq n-1}$ in \eqref{eq:alphai}.
Finally, to obtain \eqref{eq:hatLambdasample}, we also make a left-endpoint approximation in the second term of our decomposition of the increments of $\Lambda\,$, leading to
\begin{equation}\label{eq:Lam_N_relationship}
    \Lambda_{t^n_{i+1}} - \Lambda_{t^n_i} = G_{t^n_i}(t^n_{i+1}) + \int_{t^n_i}^{t^n_{i+1}}\,K(t^n_{i+1} - r)\,\left(N_{t^n_{i+1}} - N_r\right)\,dr \approx \alpha^n_i + k_0^n \,\widehat N^n_{i,i+1} = \widehat \Lambda^n_{i,i+1}\,.
\end{equation}
We will see in Section \ref{sect:convergence} that this discretization amounts to replacing the memory kernel $K$ by a measure-valued one, built on integrated quantities, hence preserving the Volterra structure of the process.

\subsection{Deriving the Inverse Gaussian parameters}\label{subsec:deriving_ig}
The Inverse Gaussian distribution comes from a short-time approximation of the characteristic functional of $\Lambda\,$. We remark that the exponent in the characteristic function of Inverse Gaussian random variables (see \ref{eq:IG_charac}) has the structure of a root of a second-order polynomial. It is a well-known fact that characteristic functionals of affine Volterra processes are given semi-explicitly, up to the solution to a Riccati--Volterra equation (see \citet[Theorem~2.2]{abi2021weak} and \citet*{abi2019affine}). Consequently, using second-order Taylor expansions in an implicit Euler fashion, we can expect to obtain second-order polynomials as small-time approximations of such equations. We detail this methodology in the present subsection to identify the right Inverse Gaussian parameters for our scheme.

From \citet[Example~2.3]{abi2021weak}\footnote{The result is proved in a more general framework which requires assumptions of the memory kernel $K\,$. However, in the special case of Hawkes processes, it can be shown to hold for any $K \in L^1([0\,,T]\,, \R_+)$ using similar arguments.}, we know that for $w \in \mathbb C$ with $\Re(w) \leq 0\,$, 
$$
\mathbb E \left[\exp\left(w\, \left(\Lambda_{t^n_{i+1}} - \Lambda_{t^n_i}\right)\right)\,\middle|\, \mathcal F_{t^n_i} \right] = \exp\left(\int_{t^n_i}^{t^n_{i+1}}\,F(\Psi(t^n_{i+1} - s))\,dG_{t^n_i}(s)\right)\,, \quad i \leq n-1\,,\quad n \geq 1\,,
$$
where $G_{t^n_i}$ is defined in \eqref{eq:G_s_t}. Here, the function $F$ is given by
\begin{equation}\label{eq:F_def}
    F(u) = w + e^{u} - 1\,, \quad u \in \mathbb C\,,
\end{equation}
and $\Psi$ is the continuous solution, with nonpositive real part, to the Volterra equation 
$$
\Psi(t) = \int_0^t\,K(t-s)\,F(\Psi(s))\,ds\,, \quad t \leq T\,.
$$
Using the notation $\Delta t = \frac{T}{n}\,$, we can make the left-endpoint approximations 
\begin{equation}\label{eq:FdG_approx}
    \int_{t^n_i}^{t^n_{i+1}}\,F(\Psi(t^n_{i+1} - s))\,dG_{t^n_i}(s) \approx F(\Psi(\Delta t))\,G_{t^n_i}(t^n_{i+1}) \approx F(\Psi(\Delta t))\,\alpha^n_i \,,
\end{equation}
along with
\begin{equation}\label{eq:Psi_approx}
    \Psi(\Delta t) \approx \int_0^{\Delta t}\,K(s) \, ds\,F(\Psi(\Delta t)) = k_0^n\,F(\Psi(\Delta t)) ,
\end{equation}
for small $\Delta t\,$. We remark that $\Psi(\Delta t)$ is here approximated by the solution to the equation $x = k_0^n\,F(x)\,$, which is not quadratic yet. Using the fact that $\Psi(0) = 0$ so that $\Psi(\Delta t)$ is small, we can approximate $F$ by its second-order Taylor expansion, given by
\begin{equation}\label{eq:F_approx}
    F(u) = w + e^u - 1\approx w + u + \dfrac{u^2}{2},
\end{equation}
for small $u\,$. Combining all of our approximations, we obtain
$$
\mathbb E \left[\exp\left(w\, \left(\Lambda_{t^n_{i+1}} - \Lambda_{t^n_{i}}\right) \right)\,\middle|\, \mathcal F_{t^n_i} \right] \approx \exp\left(\widehat \Psi(\Delta t) \frac{\alpha^n_i}{k_0^n} \right) \,,
$$
where $\widehat \Psi(\Delta t)$ is the zero root with nonpositive real part of the second order polynomial equation
$$
\frac{k_0^n}{2}\,\left(\widehat \Psi(\Delta t)\right)^2 - (1 - k_0^n)\,\widehat \Psi(\Delta t) + w\,k_0^n = 0 \,.
$$
After computation, we obtain 
$$
\widehat \Psi(\Delta t) = \frac{1 - k_0^n}{k_0^n}\, \left(1 - \sqrt{1 - 2\,\left( \frac{k_0^n}{1 - k_0^n}\right)^2\,w} \right)\,,
$$
which corresponds to the characteristic functional of an Inverse Gaussian random variable (see Appendix \ref{App:IG}) {$\xi^n_i$} with parameters $\left( \frac{\alpha^n_i}{1 - k_0^n}\,, \left(\frac{\alpha^n_i}{k_0^n} \right)^2\right)\,$.
Finally, since $N$ is a counting process with integrated intensity $\Lambda\,$, conditional on $\Lambda_{t^n_{i+1}}- \Lambda_{t^n_i}$, we have
$$
N_{t^n_{i+1}} - N_{t^n_i} \sim \mathcal P \left(\Lambda_{t^n_{i+1}}- \Lambda_{t^n_i}\right) {\approx \mathcal P \left(\xi^n_i\right) }\,,
$$
where $\mathcal P(\lambda)$ denotes the Poisson distribution with parameter $\lambda\,$. 

It is worth noting that we use equation \eqref{eq:hatLambdasample} to compute the increment $\widehat\Lambda^n_{i,i+1}$, rather than setting $\widehat\Lambda^n_{i,i+1} = \xi_{i}^n$, as might be naturally suggested by the iVi scheme for Volterra processes introduced by \citet{jaber2025simulating}. This adjustment is due to the discrete nature of point processes: since $N^n_{i,i+1}$ must be an integer, it can no longer be directly extracted from the equation
\[
\alpha^n_i + k_0^n \widehat N^n_{i,i+1}  = \widehat \Lambda^n_{i,i+1},
\]
which arises from the discretization \eqref{eq:Lam_N_relationship}. Instead, after simulating $\widehat N^n_{i,i+1}$ as a Poisson random variable with the Inverse Gaussian parameter $\xi_{i}^n$, $\widehat\Lambda^n_{i,i+1}$ is determined from the equation above, leading to \eqref{eq:hatLambdasample} and ensuring the simulation of \emph{consistent} trajectories for both the jump process $N$ and its integrated intensity $\Lambda$. 
By consistency, we understand the Volterra-type relation between the discretized processes $\Lambda^n$ and $N^n$ (see Lemma~\ref{lemma:Volterra}) that is a crucial ingredient of the proof of convergence. In particular, this helps us avoid situations in which the simulated counting process $N^n$ has no jumps (all sampled Poisson random variables are null) and $\Lambda^n > \int_0^{\cdot}\,g_0(s)\,ds\,$, which should not be possible for a Hawkes process due to \eqref{eq:integrated_intensity}.

Combining the above, we arrive at Algorithm \ref{algo:ivi_hawkes}.

\section{Weak convergence of the Hawkes iVi scheme}

\label{sect:convergence}
In this section, we formulate a convergence result and sketch the main steps of its proof. The technical details are postponed to Section~\ref{sect:convergence_proof}. We assume that $g_0\,,K \in L^1([0\,,T]\,, \R_+)$ and we denote by $n_0$ the smallest natural $n$, such that $k_0^n < 1$, so that the scheme in Algorithm~\ref{algo:ivi_hawkes} is well-defined. For the sake of notational simplicity {only}, we will {assume that} $n_0 = 1$, but we stress that all the results hold for general $n_0$. In particular, our proof of convergence does not require $\int_0^T \,K(s)\,ds < 1\,$.


    We consider the piecewise constant non-decreasing processes $\Lambda^n$ and $N^n$ defined in \eqref{eq:def_process_Lambda} and \eqref{eq:def_process_N} respectively.
We assume to all our random variables are defined on the same complete probability space $\left(\Omega, \mathcal F , \mathbb P\right)$ and we introduce the filtration 
\begin{equation}
\label{eq:def_filtration}
\mathcal F^n_t := \sigma\left(\widehat{\mathcal{F}}^n_t\, \bigcup \, \mathcal F_0\right)\,, \quad t \leq T\,, \quad n \geq 1\,,
\end{equation}
where 
$$
\widehat{\mathcal{F}}^n_t := \sigma\left(\widehat N^n_{i,i+1}\,, 0 \leq i \leq \lfloor nt /T \rfloor -1\right)\,, \quad t \leq T\,, \quad n \geq 1\,, 
$$
and $\mathcal F_0$ is the set of all $\mathbb P$-null elements of $\mathcal F\,$. For each $n \geq 1\,$, the filtration $\left(\mathcal F^n_t\right)_{t \leq T}$ is the one generated by the process $ N^n$ and satisfies the usual conditions. Note that both processes $\Lambda^n$ and $N^n$ are adapted to the filtration.

We first mention two convenient properties of our scheme, that are pivotal in the proof of weak convergence. Firstly, for each $n \geq 1\,$, the process $N^n - \Lambda^n$ is shown to be a martingale in Lemma \ref{lemma:martingales}, which is enjoyable since $N - \Lambda$ is a martingale in the limit. However, it is important to note that $\Lambda^n$ is not a predictable process, since at each step $\widehat \Lambda^n_{i,i+1}$ is set according to \eqref{eq:hatLambdasample} after sampling $\widehat N^n_{i,i+1}\,$. Second, the relation between $N^n$ and $\Lambda^n$ can also be rewritten as a stochastic Volterra equation, but with a measure-valued kernel (see Lemma \ref{lemma:Volterra}). More precisely, introducing the finite non-negative measure $K^n := \sum_{i = 0}^{n-1}\,k_i^n\,\delta_{t^n_i}$ on $[0\,,T]\,$, where the $\left(k^n_i\right)_{i \leq n-1}$ are defined in \eqref{eq:kij} and $\delta_{t}$ corresponds to the Dirac measure in $t\,$, we have
$$
\Lambda^n_t = \int_0^{\lfloor \frac{nt}{T} \rfloor \frac{T}{n}}\,g_0(s)\,ds + \int_{[0,t]}\,N^n_{t-s}\,K^n(ds)\,, \quad t \leq T\,.
$$
This relation explicits how our discretization procedure (see Subsection~\ref{subsec:dynamic_approx}) transforms the Volterra structure of \eqref{eq:integrated_intensity}.

Combining these properties with estimates on the processes, we prove the weak convergence of the scheme, which is given by the following theorem.
\begin{theorem}
\label{theorem:convergence}
    We have the weak convergence
    $$
    \left(\Lambda^n\,,N^n\right) \overset{n \to \infty}{\implies} \left(\Lambda, N\right)\,,
    $$
    in the Skorokhod $J_1$-topology, where $N$ is a Hawkes process on $[0\,,T]\,$ with exogenous intensity $g_0$ and memory kernel $K$, and $\Lambda$ is its integrated intensity.
\end{theorem}
\begin{proof}
Lemma \ref{lemma:tightness} gives the tightness of the couple of processes. Lemma \ref{lemma:characterization_limit} then characterizes any weak accumulation point $(\Lambda\,, N)\,$, showing that $N$ is a counting process (restricted to $[0\,,T]$) with integrated intensity $\Lambda = G_0 + K * N\,$.  The uniqueness in law follows from the first lemma in \citet{hawkes1974cluster}, with $H_0 = \emptyset$ in our case. Since all accumulation points have the same law, we obtain weak convergence of the sequence.
\end{proof}

\section{Numerical results}\label{sect:numerics}
In this section, we compare the iVi Hawkes scheme with two popular approaches for simulating Hawkes processes, the population-based method (see, for instance, \citet{Moller2005}) and the thinning method introduced by \citet{Ogata81}, both in terms of the marginal laws of the processes $N$ and $\Lambda$, and the distributions of the arrival times.

\subsection{Population approach and Ogata's algorithm}
We begin by recalling the ideas behind these two methods and highlighting their respective advantages and disadvantages.

\paragraph{Population approach}
The population approach \citep{Moller2005} models the Hawkes process as a superposition of non-homogeneous Poisson processes of two types:
\begin{enumerate}
    \item \textit{``Migrants''}, or arrival times generated by the Poisson component $g_0$ of the intensity.
    \item \textit{``Descendants''}, i.e., arrivals triggered by previous arrivals (of both types) corresponding to the terms $K(t - \tau), \, t \geq \tau$ of the intensity for a given arrival time $\tau$.
\end{enumerate}
This decomposition reveals the tree structure of the Hawkes process and reduces the simulation problem to simulating non-homogeneous Poisson processes applied sequentially to successive ``generations'' of arrival times. Simulation of a non-homogeneous Poisson process is straightforward when the inverse integrated intensity is available, thanks to the following result (cf. e.g., \citet{brown1988}).

\begin{theorem}[Random time change]\label{thm:time_change}
    Let $(\tau_1, \ldots, \tau_{N_T})$ be a realization of a point process with integrated intensity $\Lambda$ over $[0, T]$ and $\Lambda(T) < \infty$ a.s. Then, the transformed arrivals $(\Lambda(\tau_1), \ldots, \Lambda(\tau_{N_T}))$ form a standard Poisson process with unit rate on $[0, \Lambda(T)]$.
\end{theorem}

Hence, to simulate a non-homogeneous Poisson process with (deterministic) integrated intensity $\Lambda$, one simply simulates a standard Poisson process on $[0, \Lambda(T)]$ and then applies $\Lambda^{-1}$ to its arrival times. Algorithm~\ref{alg:population} summarizes this procedure. We use the notation $\bar K(t) := \int_0^t K(s)\,ds$ for the integrated kernel.

\begin{algorithm}[H]
\caption{\textbf{Population approach}}\label{alg:population}
\begin{algorithmic}[1]
\State \textbf{Input:} $G_0, G_0^{-1}, \bar K, \bar K^{-1}$.
\State \textbf{Output:} event times $\mathcal{T} = (\tau_i)_{i=1\ldots, N_T}$.
\State Initialize empty sets: parent queue $\mathcal{P} \gets \varnothing$, output set $\mathcal{T} \gets \varnothing$.
\State Simulate ``migrants'' as an inhomogeneous Poisson process on $[0,T]$ with integrated intensity $G_0$, yielding arrivals $\tau_1^0, \ldots, \tau_{n_0}^0$, where $n_0$ denotes the total number of arrivals on $[0, T]$.
\State Add $\{\tau_1^0, \ldots, \tau_{n_0}^0\}$ to $\mathcal{P}$.
\While{$\mathcal{P} \neq \varnothing$}
    \State Remove the first element $\tau$ from $\mathcal{P}$.
    \State Simulate arrivals $\tilde\tau_1^\tau, \ldots, \tilde\tau_{n_\tau}^\tau$ of a standard Poisson process on $[0, \bar K(T-\tau)]$.
    \State Transform to descendant event times: $\tau_i^\tau \gets \tau + \bar K^{-1}(\tilde\tau_i^\tau)$ for $i = 1,\ldots,n_\tau$.
    \State Add $\{\tau_1^\tau, \ldots, \tau_{n_\tau}^\tau\}$ to $\mathcal{P}$
    \State Add $\tau$ to $\mathcal{T}$.
\EndWhile
\State Sort $\mathcal{T}$ in ascending order.
\end{algorithmic}
\end{algorithm}

\paragraph{Ogata's algorithm}
The algorithm introduced by \citet{Ogata81} generalizes the standard thinning algorithm. The main idea is as follows: if the kernel $K$ is finite and decreasing, and $g_0$ is bounded, then for $t \in [0, T]$, given $\mathcal{F}_t$, one can bound the intensity of the Hawkes process between $t$ and the next arrival by
\[
M = \max\limits_{u \in [t, T]} \{g_0(u)\} + \sum\limits_{\tau_i \leq t} K(t - \tau_i),
\]
and apply the standard thinning method (cf.~\citet{Lewis1979}). Here, $(\tau_i)_{i = 1, \ldots, N_t}$ denotes the previously simulated arrivals. We summarize Ogata's method in Algorithm~\ref{alg:ogata}. 

\begin{algorithm}[H]
\caption{ \textbf{- Ogata's algorithm}}\label{alg:ogata}
\begin{algorithmic}[1]
\State \textbf{Input:} \(g_0, K\) with $\max_{u \in [0, T]}g_0(u) < \infty$ and decreasing $K$, such that $K(0) < \infty$.
\State \textbf{Output:} \(\mathcal{T} = (\tau_i)_{i = 1, \ldots, N_T}\).
\State Set $t \gets 0$.
\While{$t < T$}
\State Set $M \gets \max\limits_{u \in [t, T]} \{g_0(u)\} + \sum\limits_{\tau_i \leq t} K(t - \tau_i)$.
\State Simulate $\xi \sim \mathrm{Exp}(1 / M)$ and $U \sim \mathcal{U}(0, 1)$
\State Set $\tau \gets t + \xi$.
\State If $\tau < T$ and $U < \frac{g_0(\tau) + \sum_{\tau_i \leq t} K(\tau - \tau_i)}{M}$, set $\tau_{i+1} \gets \tau$.
\State Set $t \gets \tau$.
\EndWhile
\end{algorithmic}
\end{algorithm}

Both methods, as presented, allow for exact simulation of the Hawkes process. However, the population approach requires knowledge of the inverse integrated kernel, which may be problematic for complex kernels like the Mittag--Leffler kernel used by \citet*{Chen2021}. In addition, this kernel and the fractional kernel $K(t) = \dfrac{c}{\Gamma(\alpha)}t^{\alpha - 1}, \alpha < 1$ violate the assumption $K(0) < \infty$ required by Ogata's algorithm. Moreover, the monotonicity assumption is violated by Erlang or, more generally, gamma kernels (see, e.g.,~\citet{Duarte}).

To handle singular kernels, \citet{Chen2021} propose shifting the kernel by a small $\epsilon > 0$ when computing $M$. In this case, Ogata's algorithm is no longer exact, though the error is almost negligible. Similarly, to overcome the monotonicity requirement, one may use any decreasing kernel $K^m \geq K$ with $K^m(0) < \infty$ to compute $M$. However, this increases the number of simulations in the acceptance–rejection procedure and thus reduces the efficiency of the algorithm.

\subsection{Numerical experiments}
\label{subsec:numerical_experiments}
For the numerical illustration, we will use the gamma kernel
$$
K_{\Gamma}(t) = ce^{-b t}\dfrac{t^{\alpha - 1}}{\Gamma(\alpha)}, \quad b > 0, \ \alpha > 0, \ c > 0, \ t \geq 0,
$$
with the parameters $b = 3, \, \alpha = 2, \, c = 0.9b^\alpha$ and $\mu = 5,\, T = 1$. Note that when $\alpha$ is an integer, the gamma kernel reduces to the Erlang kernel widely used in the applications of the Hawkes processes.

Figure~\ref{fig:mean} shows the sample trajectories of $N$, $\Lambda$ and the instantaneous intensity $\lambda$ generated with the iVi scheme with time step $0.001$. {The instantaneous intensity is computed using the left-endpoint approximation
$$
\lambda_{t_i} = g_0(t_i) + \int_0^{t_i}K(t_i-s)\,dN_s  = g_0(t_i) + \sum_{j=0}^{i-1}\int_{t_j}^{t_{j+1}}K(t_i-s)\,dN_s \approx g_0(t_i) + \sum_{j=0}^{i-1}K(t_i-t_j)\widehat N_{j, j+1}^n,
$$
for $i = 0, \ldots, n$.}
\begin{figure}[H]
    \begin{center}
    \includegraphics[width=1\linewidth]{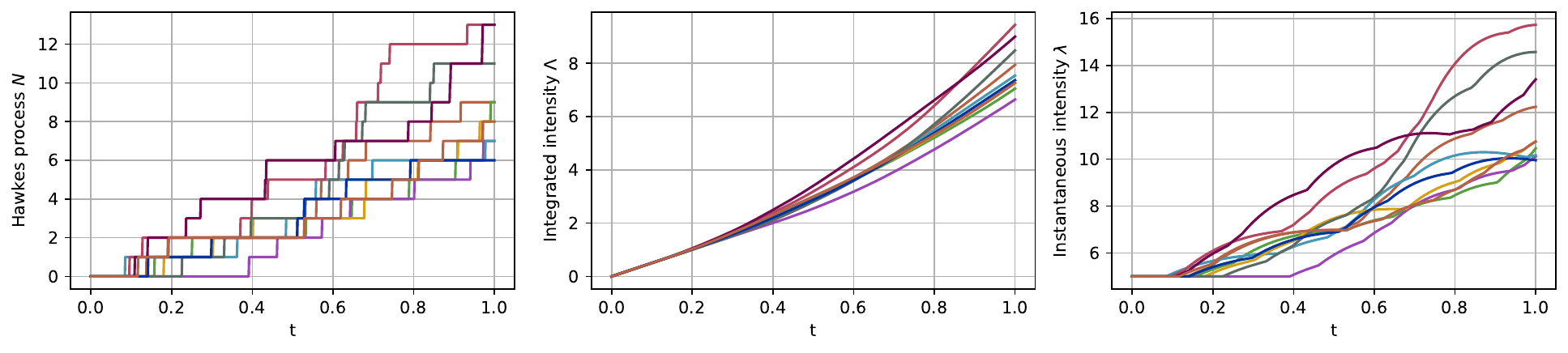}
    \caption{Sample trajectories simulated with the iVi scheme with time step $0.001$.}
    \label{fig:mean}
    \end{center}
\end{figure}
To illustrate the convergence of the Hawkes iVi scheme, we will examine:
\begin{enumerate}
\item The marginal laws of $N_T$ and $\Lambda_T$.
\item The distribution of the arrival times $\tau_1, \ldots, \tau_{N_T}$.
\end{enumerate}

\paragraph{Laplace transforms.}
To illustrate the convergence of the marginal laws, we compute the Laplace transforms
$
    \phi_{N_T}(w) = \exp\left(w N_T\right)
$
and
$
    \phi_{\Lambda_T}(w) = \exp\left(w \Lambda_T\right)
$
for different numbers of time steps. The convergence of the Monte Carlo estimator of the characteristic function, based on $10^6$ trajectories and with $w = -\dfrac{1}{\mathbb{E}[N_T]}$, is shown in Figure~\ref{fig:cf_conv}. The horizontal bar indicates three standard deviations of the estimator. Figure~\ref{fig:error_sims} reports the computation time required to simulate $10^5$ trajectories, together with the corresponding estimation error, for the iVi scheme under different discretization steps. The horizontal bars indicate the simulation time of the exact methods. These results demonstrate that the iVi scheme can provide substantial acceleration over Ogata's algorithm and the population-based approach while maintaining high precision.

\begin{figure}[H]
    \begin{center}
    \includegraphics[width=1\linewidth]{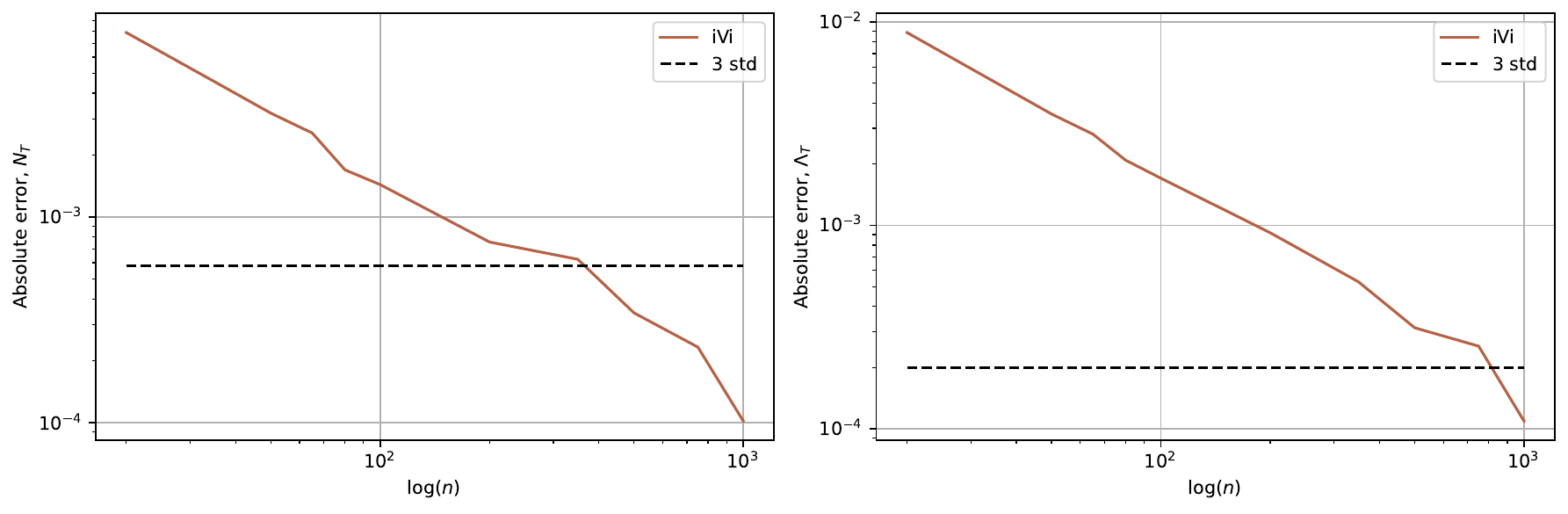}
    \caption{Convergence of the Laplace transform with $w = -\frac{1}{\mathbb{E}[N_T]}$ for $N_T$ (left) and $\Lambda_T$ (right). The $x$-axis shows the number of time steps, and the $y$-axis shows the absolute error of the Monte Carlo estimator.}
    \label{fig:cf_conv}
    \end{center}
\end{figure}

\begin{figure}[H]
    \begin{center}
    \includegraphics[width=1\linewidth]{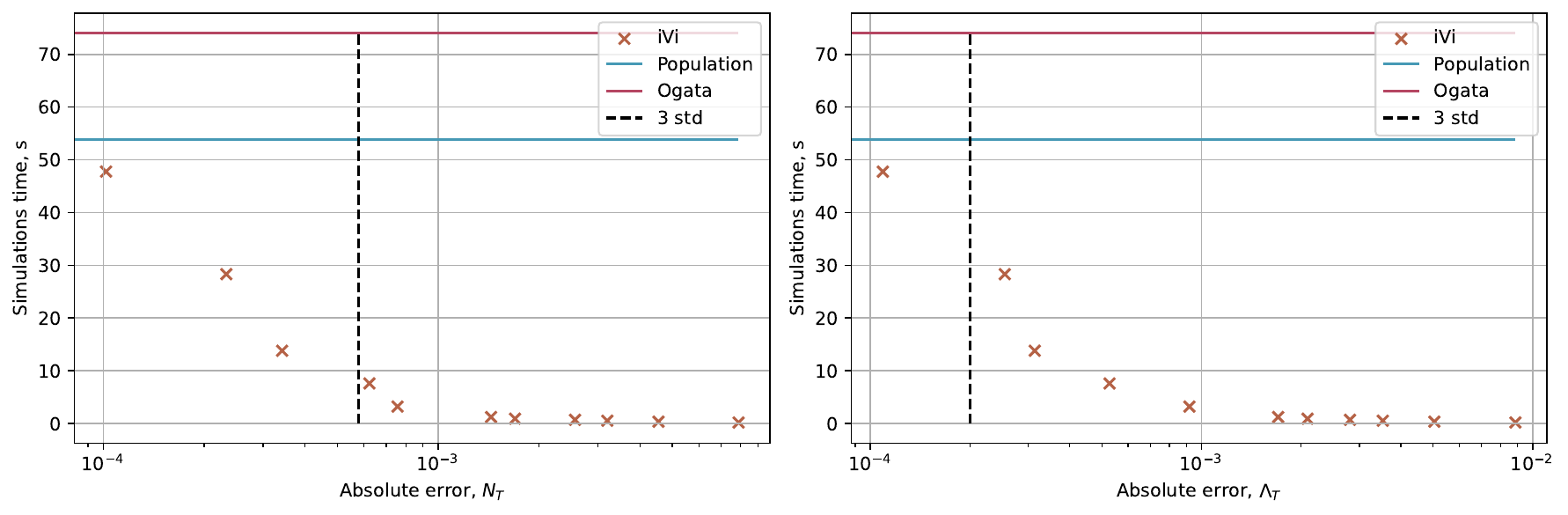}
    \caption{Simulation time for $10^5$ trajectories ($y$-axis) versus the absolute error of the Monte Carlo estimator of the Laplace transform ($x$-axis). Crosses represent estimators obtained with the iVi scheme under different discretization steps. Horizontal bars indicate the exact methods: Population (blue) and Ogata (purple).}
    \label{fig:error_sims}
    \end{center}
\end{figure}

\paragraph{Comparison of marginal laws.}
The proximity of the exact marginal laws $\mathcal{L}(N_T)$ and $\mathcal{L}(\Lambda_T)$ to their counterparts generated by the iVi scheme is assessed by simulating $N = 10^4$ trajectories with each scheme and comparing the empirical distributions of $N_T$ and $\Lambda_T$ with those obtained from the exact population approach. In particular, we plot both the empirical distributions and the corresponding Q–Q plots against samples generated by the exact method. For the iVi scheme, the Kolmogorov–-Smirnov test for $\Lambda_T$ yields a $p$-value equal to $0.92$. The empirical distributions of $N_T$ and $\Lambda_T$ are shown in Figures~\ref{fig:marginal_N} and \ref{fig:marginal_U}, respectively.

\begin{figure}[H]
    \begin{center}
    \includegraphics[width=1\linewidth]{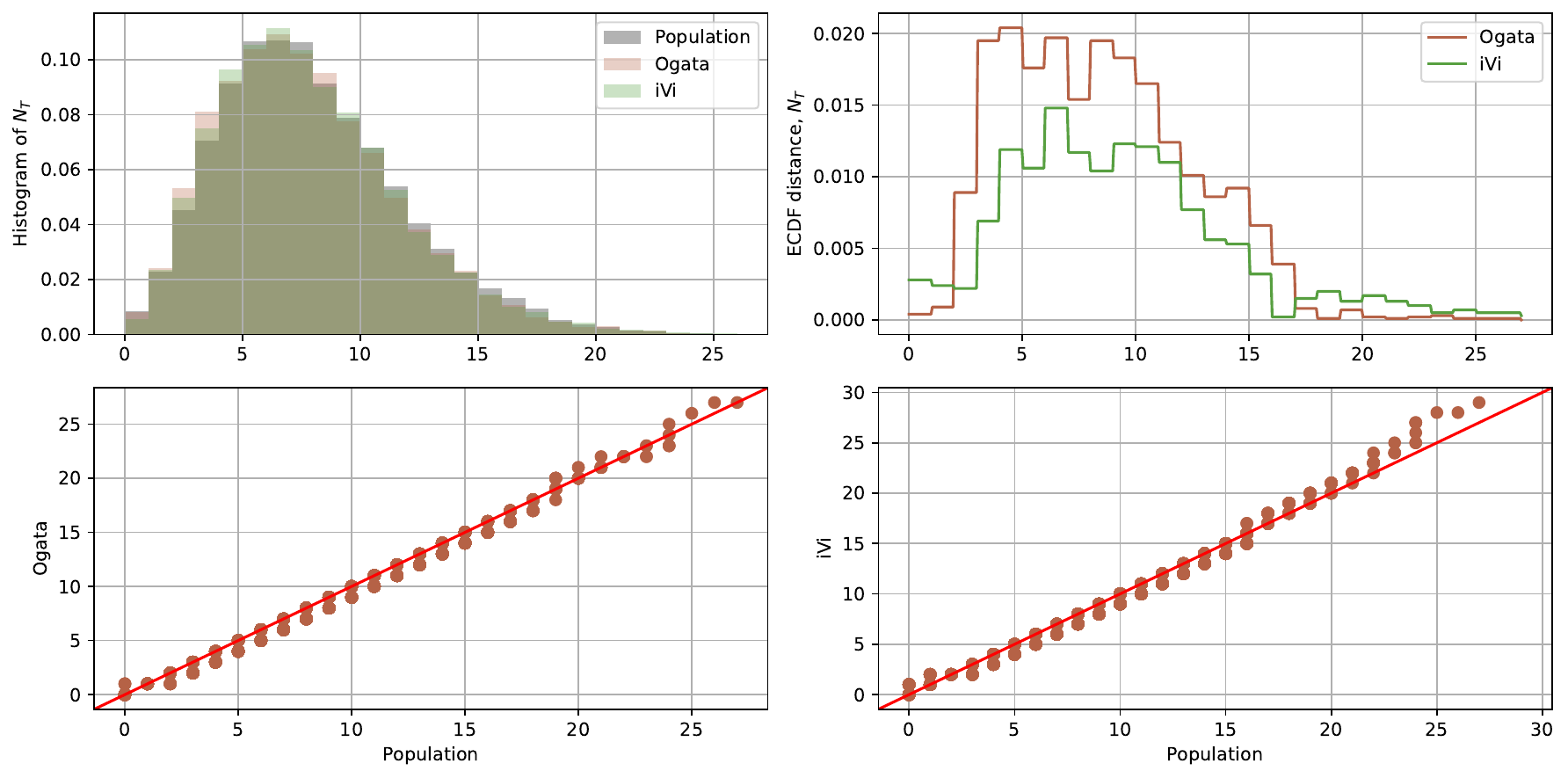}
    \caption{Marginal law of $N_T$. Empirical densities (top left), distances to the empirical CDF of the population method (top right), Q-Q plots for Ogata algorithm (bottom left) and iVi scheme (bottom right). A time step $0.01$ was taken for the iVi scheme.}
    \label{fig:marginal_N}
    \end{center}
\end{figure}

\begin{figure}[H]
    \begin{center}
    \includegraphics[width=1\linewidth]{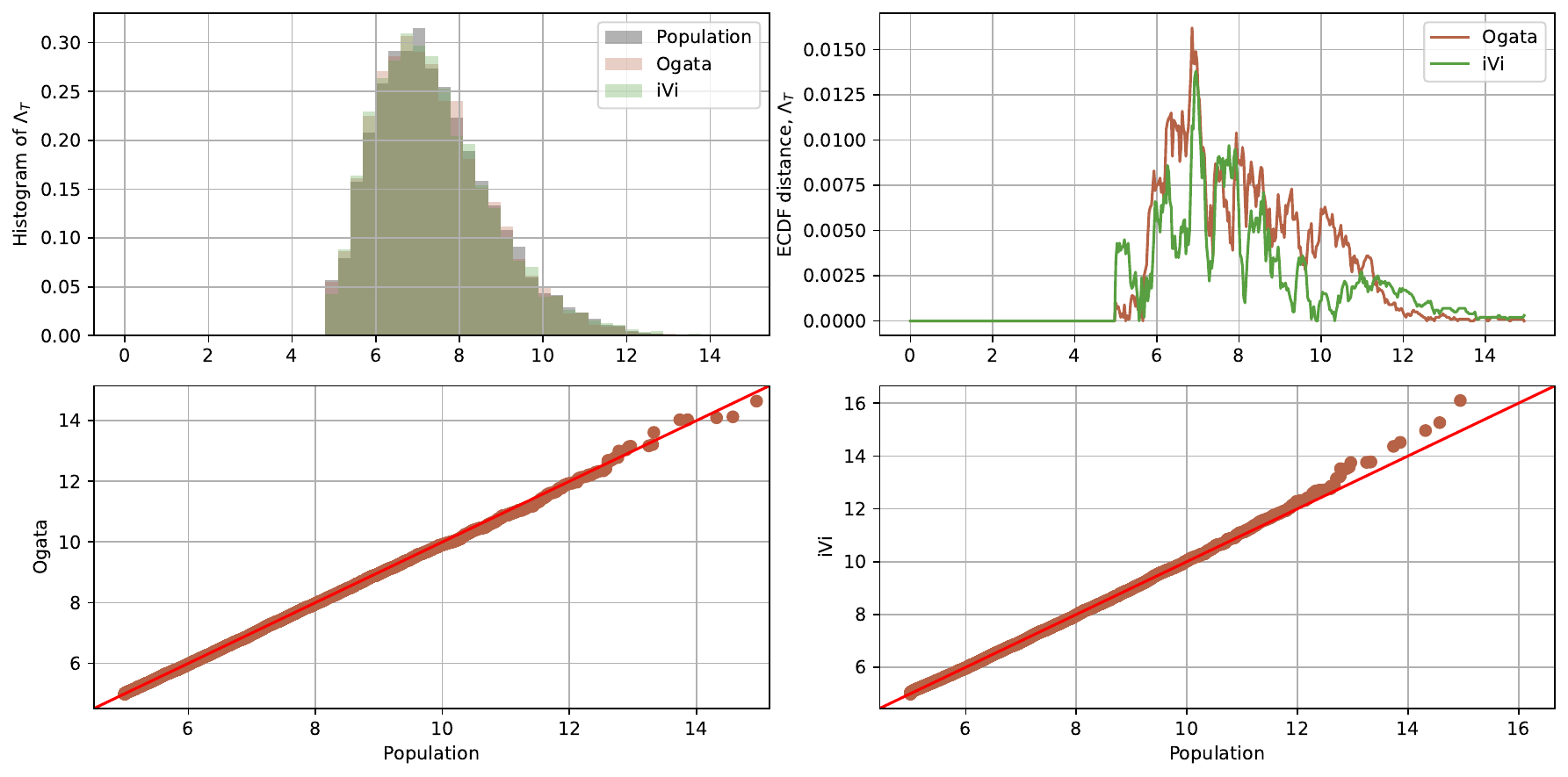}
    \caption{Marginal law of $\Lambda_T$. Empirical densities (top left), distances to the emperical CDF of the population method (top right), Q-Q plots for Ogata algorithm (bottom left) and iVi scheme (bottom right). A time step $0.01$ was taken for the iVi scheme.}
    \label{fig:marginal_U}
    \end{center}
\end{figure}

\paragraph{Distribution of arrival times.}  
In the iVi scheme, the arrival times are simulated uniformly based on the simulated increments $\widehat N_{i, i+1}^n$ over each discretization interval. Essentially, this means that the instantaneous jump intensity $\lambda$ is approximated on each interval $[t_i, t_{i+1}]$ by 
$$
\widehat{\lambda}_i^n = \dfrac{\widehat\Lambda^n_{i, i+1}}{t^n_{i+1} - t^n_{i}}.
$$
To assess the quality of this approximation, we perform standard time--change tests, see, e.g.,~\citet{ogata1988statistical}. Specifically, we simulate the arrival times $\mathcal{T}^n = (\tau_i^n)_{i = 1, \ldots, N_T^n}$, compute  
\[
\tau^{n,*}_i = \Lambda_{\tau_i^n}, \quad i = 1, \ldots, N_T^n,
\]  
using \eqref{eq:integrated_intensity}, and verify that $(\tau^{n,*}_i)_{i = 1, \ldots, N_T^n}$ follow the law of a standard Poisson process (by Theorem~\ref{thm:time_change}). To this end, we compare the empirical distribution of the interarrival times $(\tau^{n,*}_{i+1} - \tau^{n,*}_i)_{i=1, \ldots, N_T^n-1}$ with the exponential distribution $\mathrm{Exp}(1)$, corresponding to the waiting times of a Poisson process. In addition, we plot the pairs $(e^{-\tau^{n,*}_i}, e^{-\tau^{n,*}_{i+1}})_{i=1, \ldots, N_T^n-1}$, which should be independent and uniformly distributed on $[0,1]$.  

For the iVi scheme with discretization step $0.1$ and $T = 10$, the Kolmogorov--Smirnov test for the interarrival times $(\tau^{n,*}_{i+1} - \tau^{n,*}_i)_{i=1, \ldots, N_T^n-1}$ yields a $p$-value of $0.43$. The empirical distribution of the transformed arrivals over $[0,10]$ is shown in Figure~\ref{fig:arrivals}.  

\begin{figure}[H]
    \begin{center}
    \includegraphics[width=1\linewidth]{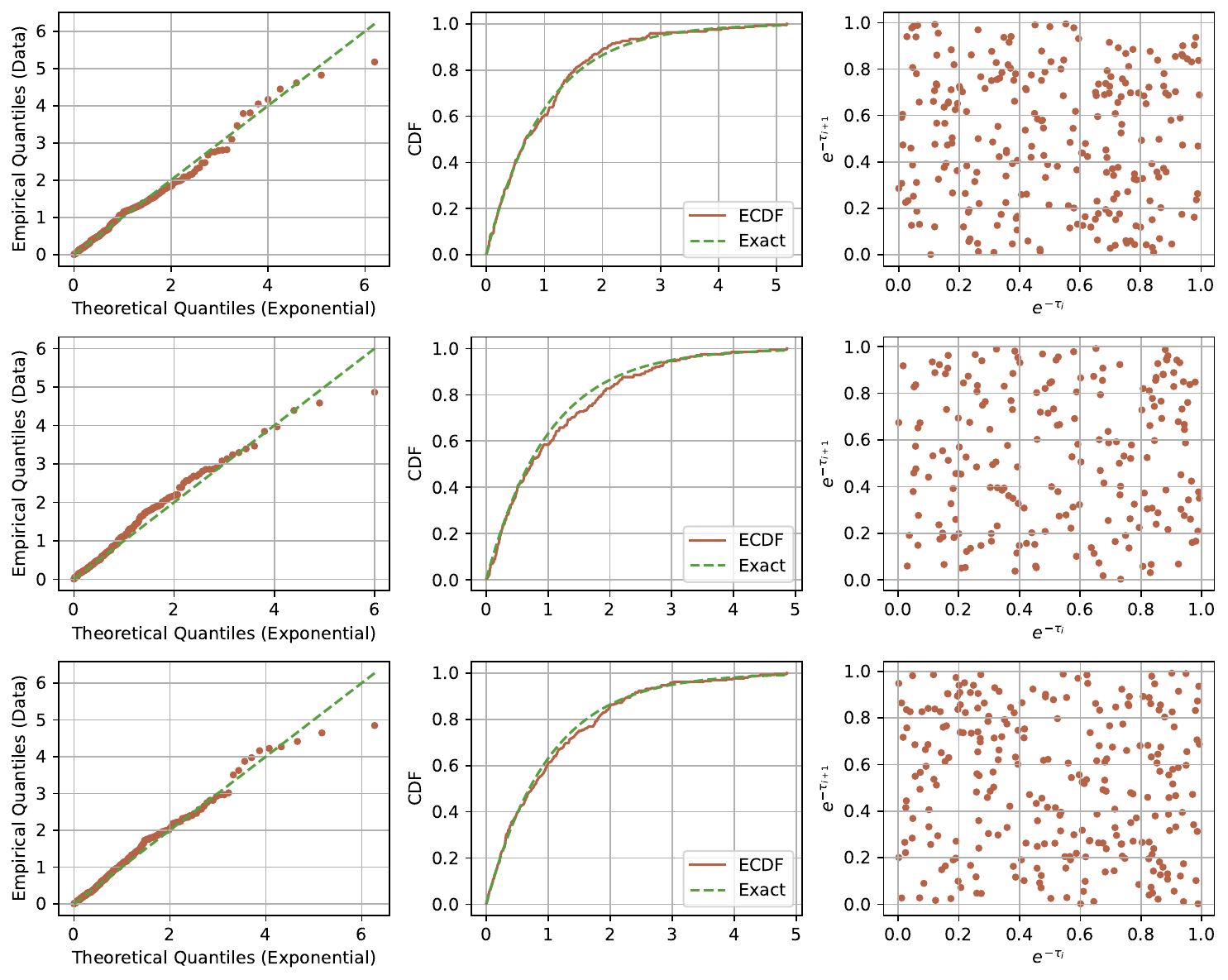}
    \caption{Transformed arrival times generated by the population approach (top), Ogata's algorithm (middle), and the iVi scheme (bottom). The columns correspond to: Q--Q plots against the exact law (left), empirical CDFs (middle), and scatter plots of $(e^{-\tau^{n,*}_i}, e^{-\tau^{n,*}_{i+1}})_{i=1, \ldots, N_T^n-1}$ (right).}
    \label{fig:arrivals}
    \end{center}
\end{figure}

\paragraph{Importance of the implicit step.}
We conclude this subsection by illustrating the relevance of the Inverse Gaussian simulation step \eqref{eq:hatNsample}. The motivation for introducing the random variable $\xi_i^n$ lies in the implicit discretization scheme for the characteristic function described in Subsection~\ref{subsec:deriving_ig}. A simpler alternative would be to use an explicit scheme: first approximate $U_{t_{i+1}} - U_{t_i} \approx \alpha_i^n$, then simulate $\widehat N^n_{i,i+1} \sim \mathcal{P}(\alpha_i^n)$, and finally update the integrated intensity using \eqref{eq:hatLambdasample}. We refer to this approach as the \textit{explicit integrated intensity scheme}.

Note that this scheme can be viewed as a refined version of the discretization method presented in \citep[Theorem~2.4]{kirchner2017estimation}, since we use the integrated kernel instead of discretizing the kernel itself. However, in high-activity regimes this scheme proves to be significantly less efficient, as we illustrate below for the fractional kernel
$$
K_H(t) = \dfrac{ct^{H - 0.5}}{\Gamma(H + 0.5)}, \quad t \geq 0,
$$
with $H = 0.1,\ c = 0.1$ and $\mu = 5$. We plot the convergence results for the Laplace transforms
\[
\phi_{N_T}(w) = \exp\left(w N_T\right) \quad \text{and} \quad 
\phi_{\Lambda_T}(w) = \exp\left(w \Lambda_T\right),
\]
with $w = -\frac{1}{\mathbb{E}[N_T]}$ and $T = 30$, in Figure~\ref{fig:cf_conv_explicit}. We observe that in this case, the iVi scheme is more accurate and exhibits a higher empirical rate of convergence than the explicit integrated intensity scheme, which justifies the use of the implicit approach we propose.

\begin{figure}[H]
    \begin{center}
    \includegraphics[width=1\linewidth]{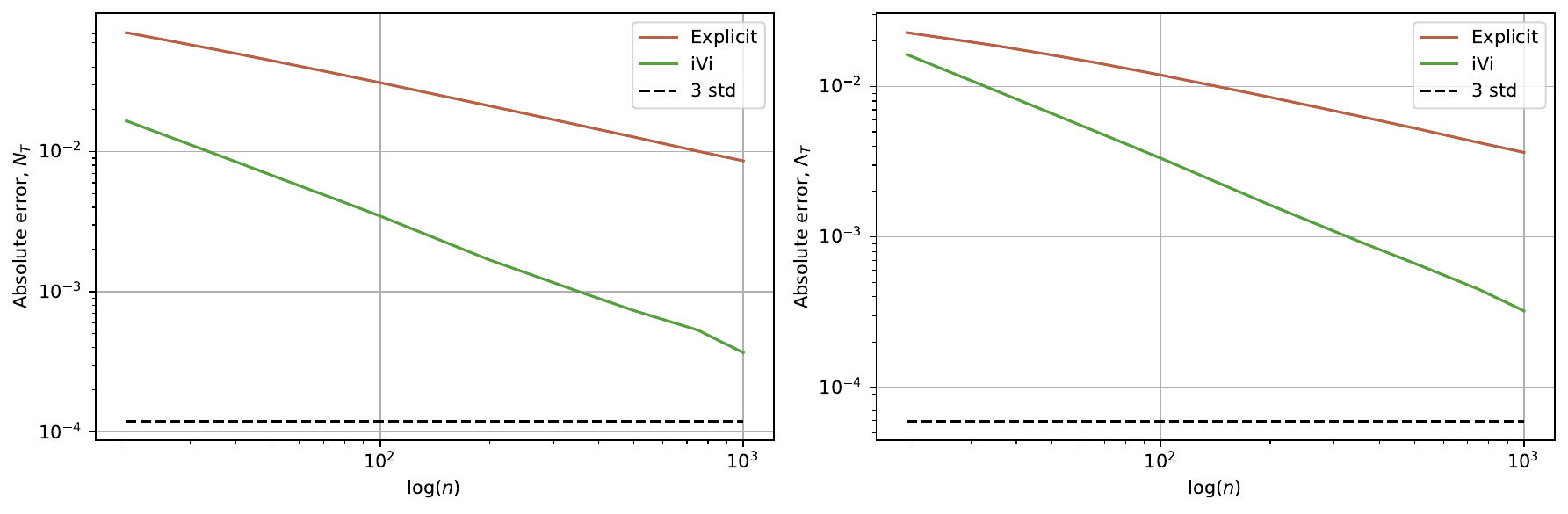}
    \caption{Convergence of the Laplace transform with $w = -\frac{1}{\mathbb{E}[N_T]}$ for $N_T$ (left) and $\Lambda_T$ (right). The $x$-axis shows the number of time steps, and the $y$-axis shows the absolute error of the Monte Carlo estimator.}
    \label{fig:cf_conv_explicit}
    \end{center}
\end{figure}

\subsection{A resolvent version of the iVi scheme}
\label{subsec:resolvent_scheme}
In this subsection, we show how the numerical performance of the iVi scheme for Hawkes processes can be significantly improved by using the resolvent of its kernel $K$, i.e., the kernel $R \in L^1([0, T], \mathbb{R})$ satisfying
\begin{equation*}
    R * K = R - K.
\end{equation*}
We refer to Appendix~\ref{app:kernels} for the precise definition and a list of commonly used kernels along with their resolvents.

Recall that the dynamics of the integrated intensity are given by
\begin{equation}\label{eq:integrated_intensity_2}
\Lambda_t = \int_0^t g_0(s)\,ds + \int_0^t K(t-s)\,N_s\,ds, \quad t \leq T.
\end{equation}

The following proposition allows us to rewrite \eqref{eq:integrated_intensity_2} in terms of the resolvent $R$.
\begin{proposition}
    The integrated intensity $\Lambda$ satisfies 
    \begin{equation}\label{eq:integrated_variance_res}
        \Lambda_t = G_0^R(t) + \int_0^t R(t-s)Z_s\,ds,
    \end{equation}
    where $Z_t := N_t - \Lambda_t$ and
    \[
    G_0^R(t) := \int_0^t g_0(s)\,ds + \int_0^t R(t - s)\int_0^s g_0(u)\,du\,ds =  G_0(t) + \int_0^t R(t - s)G_0(s)\,ds.
    \]
\end{proposition}

\begin{proof}
Note that equation \eqref{eq:integrated_intensity_2} can be rewritten as a convolution:
\begin{equation*}
    \Lambda = G_0 + K * N,
\end{equation*}
where $G_0(t) = \int_0^t g_0(u)\,du$. Convolving both sides with $(\mathrm{Id} + R)$ on the left and applying the resolvent equation yields:
\begin{align*}
    \Lambda + R * \Lambda &= G_0 + R * G_0 + K * N + R * K * N \\
    &= G_0 + R * G_0 + R * N.
\end{align*}
It remains to observe that $G_0^R = G_0 + R * G_0$, which concludes the proof.
\end{proof}

The equation \eqref{eq:integrated_variance_res} suggests the following modification of the iVi scheme.

\begin{algorithm}[H]\label{algo:ivi_res}
\caption{ \textbf{- The Hawkes Resolvent iVi scheme:} Simulation of \( \widehat \Lambda\,, \widehat N \)}\label{alg:simulation_res}
\begin{algorithmic}[1]
\State \textbf{Input:} \( G_0^R, R \), uniform partition of $[0,T]$, \( t_i^n = i T/n  \), with $i=0,\ldots, n$.
\State \textbf{Output:} \(  \widehat \Lambda^n_{i, i+1}\,, \widehat N^n_{i, i+1} \) for \( i = 0, \ldots, n-1 \) and (optional) $\mathcal{T}^n = (\tau_k^n)_{k=1, \ldots, N_T^n}$.
\State Compute
\begin{align}\label{eq:kij_res}
r_{j}^n = \int_{t_j^n}^{t_{j+1}^n} R\left(s\right) \, ds, \quad {j=0,\ldots, n-1}.
\end{align}
\For{$i = 0$ to $n-1$}
\State Compute the increment:
\begin{align}\label{eq:alphai_res}
    \alpha_i^n  = \max\left\{(G_0^R(t^n_{i+1}) - G_0^R(t^n_{i})) + \sum_{j=0}^{i-1}   r^n_{i-j}\,\widehat Z^n_{j,j+1},\, 0 \right\}\,,  
\end{align}
with  $\alpha^n_0=G_0^R(t^n_{1}),$  for $i=0$\,.
    \State Simulate the increment:
\begin{align}
    \label{eq:hatNsample_res}
    \widehat N^n_{i,i+1} &\sim \mathcal P\left(\xi^n_{i}\right), \quad\text{where} \quad \xi^n_{i} \sim IG\left({\alpha_i^n},\left(\frac{\alpha_i^n}
    {r_{0}^n}\right)^2 \right).
\end{align}
    \State Compute the increments:
\begin{align}
    \widehat Z^n_{i,i+1} &= \widehat N^n_{i,i+1} - \xi^n_{i}, \\
    \widehat\Lambda^n_{i, i+1} &= \dfrac{\alpha_i^n + r_0^n \widehat N^n_{i,i+1}}{1 + r_0^n}. \label{eq:lam_inc_res}
\end{align}
\State {Optional step (jump times):} Simulate $\widehat N^n_{i,i+1}$ i.i.d. uniform random variables on $[t_i, t_{i+1})$, sort, and append them to $\mathcal{T}^n$.
\EndFor

\end{algorithmic}
\end{algorithm}

The motivation for using the resolvent version of the scheme lies again in the characteristic function of $\Lambda$. Indeed, as in Section~\ref{sect:scheme_motivation}, we have,
for $w \in \mathbb C$ with $\Re(w) \leq 0\,$, 
$$
\mathbb{E} \left[\exp\left(w\, \Lambda_{t^n_i, t^n_{i+1}}\right)\,\big|\, \mathcal{F}_{t^n_i} \right] = \exp\left(\int_{t^n_i}^{t^n_{i+1}}\,F^R(\Psi(t^n_{i+1} - s))\,dG_{t^n_i}^R(s)\right), \quad i \leq n-1,\quad n \geq 1,
$$
where
\begin{equation}\label{eq:F_R_def}
    F^R(u) = w + e^{u} - 1 - u, \quad u \in \mathbb{C},
\end{equation}
$\Psi$ is the solution, with nonpositive real part, to the Volterra equation
$$
\Psi(t) = \int_0^t R(t-s)\,F^R(\Psi(s))\,ds, \quad t \leq T,
$$
and $G_{t^n_i}^R$ is given by
\begin{equation*}
    G_{t^n_i}^R(s) = G_0^R(s) - G_0^R(t^n_i) + \int_{t^n_i}^s \int_0^{t^n_i} R(u-r)\,d(N_r - \Lambda_r), \quad s \geq t^n_i,\quad i \leq n-1,\quad n \geq 1.
\end{equation*}

Approximating $F^R$ by a quadratic function gives
\begin{equation}\label{eq:F_R_approx}
    F^R(u) = w + e^{u} - 1 - u \approx w + \frac{u^2}{2}, \quad u \in \mathbb{C}.
\end{equation}
Comparing \eqref{eq:F_approx} and \eqref{eq:F_R_approx} explains why the left-endpoint approximations
\[
\int_{t^n_i}^{t^n_{i+1}} F^R(\Psi(t^n_{i+1} - s))\,dG^R_{t^n_i}(s) \approx F^R(\Psi(\Delta t))\,G_{t^n_i}^R(t^n_{i+1}),
\]
and
\[
\Psi(\Delta t) \approx \int_0^{\Delta t} R(s)\,ds\,F^R(\Psi(\Delta t)) = r_0^n\,F^R(\Psi(\Delta t)),
\]
are more accurate than their counterparts in the iVi scheme given by \eqref{eq:FdG_approx} and \eqref{eq:Psi_approx}. Indeed, while the function $F$ defined in \eqref{eq:F_def} behaves like $F(u) \sim w + u$, the function $F^R$ is approximately quadratic: $F^R(u) \sim w + \frac{u^2}{2}$ as $u \to 0$, which increases the approximation order. We will show that this dramatically improves the speed of convergence.

The next steps of the proof are similar to those in Section~\ref{sect:scheme_motivation} and lead to Algorithm~\ref{algo:ivi_res}.

Again, as in the scheme of Algorithm~\ref{algo:ivi_hawkes}, 
the increment $\widehat{\Lambda}^n_{i,i+1}$ is computed using \eqref{eq:lam_inc_res}, 
which follows from the relation
\[
\Lambda_{t^n_{i+1}} - \Lambda_{t^n_i} 
\approx \alpha^n_i + r_0^n \bigl(\widehat{N}^n_{i,i+1} - \widehat{\Lambda}^n_{i,i+1}\bigr) 
= \widehat{\Lambda}^n_{i,i+1},
\]
arising from a discretization similar to \eqref{eq:Lam_N_relationship}, see the discussion at the end of Subsection~\ref{subsec:deriving_ig}.


Importantly, for the computation of $\alpha_i^n$ in \eqref{eq:alphai_res}, one must use 
\[
\widehat Z^n_{i,i+1} = \widehat N^n_{i,i+1} - \xi^n_{i},
\]
and not 
\[
\widehat Z^n_{i,i+1} = \widehat N^n_{i,i+1} - \widehat\Lambda^n_{i,i+1},
\]
since the latter choice significantly deteriorates the performance of the scheme.  
However, this leads to a potential loss of positivity in the discretization of \eqref{eq:alphai_res}. Although the integral $G_{t^n_i}^R(t_{i+1}^n)$ is positive, its discretization may not preserve positivity because the increments $\widehat Z^n_{i,i+1}$ are no longer guaranteed to be positive and are difficult to control. While the capping in \eqref{eq:alphai_res} does not affect the numerical results, it complicates the theoretical analysis of convergence.

Another inconvenience of the scheme is the need to compute the resolvent of the kernel $K$, which may be a challenging problem in some cases. However, when the resolvent is computable, the resolvent iVi scheme exhibits significantly better performance than the standard scheme given by Algorithm~\ref{algo:ivi_hawkes}, as follows from the numerical experiments for the fractional kernel
$$
K_H(t) = \dfrac{ct^{H - 0.5}}{\Gamma(H + 0.5)}, \quad t \geq 0,
$$
with $H = 0.1,\ c = 0.1$ and $\mu = 5$. The convergence results for the Laplace transforms
\[
\phi_{N_T}(w) = \exp\left(w N_T\right) \quad \text{and} \quad 
\phi_{\Lambda_T}(w) = \exp\left(w \Lambda_T\right),
\]
with $w = -\frac{1}{\mathbb{E}[N_T]}$ and $T = 30$, as well as the corresponding execution times, are shown in Figures~\ref{fig:cf_conv_res} and \ref{fig:sim_times_res}.  

For this example, the resolvent scheme achieves a precision of three standard deviations already with $80$ time steps. In this case, the execution time for the iVi Resolvent scheme is approximately $1100$ times faster than for the Population approach. Such a substantial difference occurs when the average number of jumps per trajectory becomes sufficiently large relative to the number of time steps needed for the iVi Resolvent scheme to reach convergence. Our experiments indicate that this is typically the case, with the iVi Resolvent scheme outperforming the exact methods across different choices of kernels and simulation horizons.

\begin{figure}[H]
    \begin{center}
    \includegraphics[width=1\linewidth]{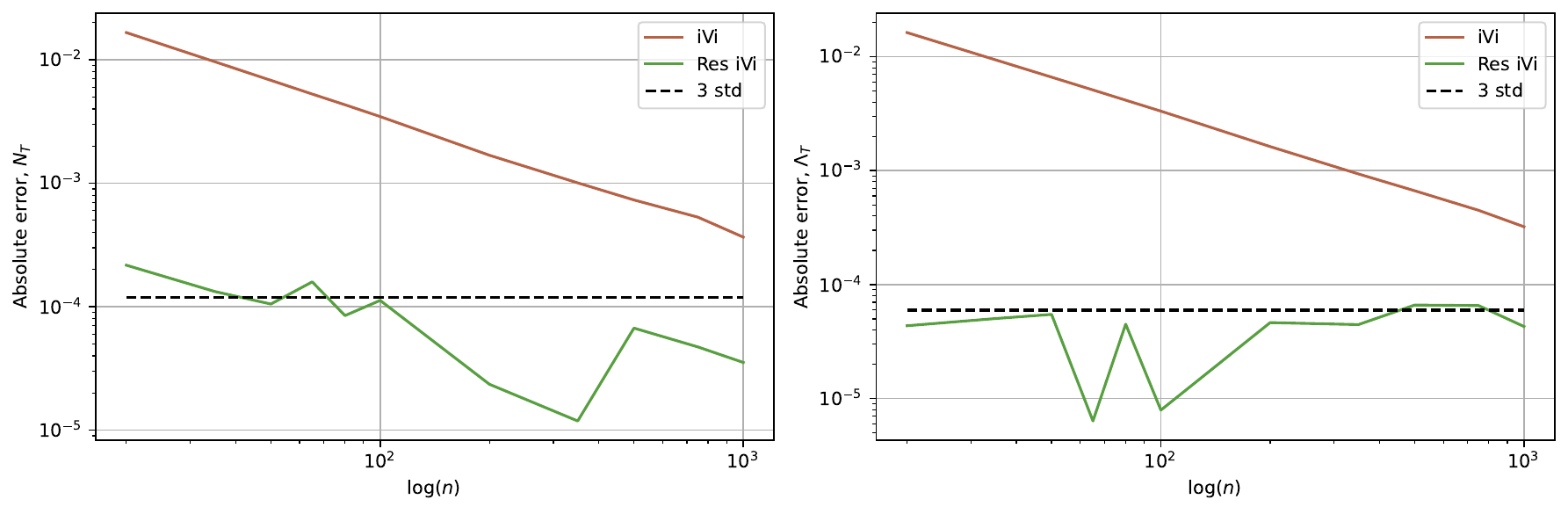}
    \caption{Convergence of the Laplace transform with $w = -\frac{1}{\mathbb{E}[N_T]}$ for $N_T$ (left) and $\Lambda_T$ (right). The $x$-axis shows the number of time steps, and the $y$-axis shows the absolute error of the Monte Carlo estimator.}
    \label{fig:cf_conv_res}
    \end{center}
\end{figure}

\begin{figure}[H]
    \begin{center}
    \includegraphics[width=1\linewidth]{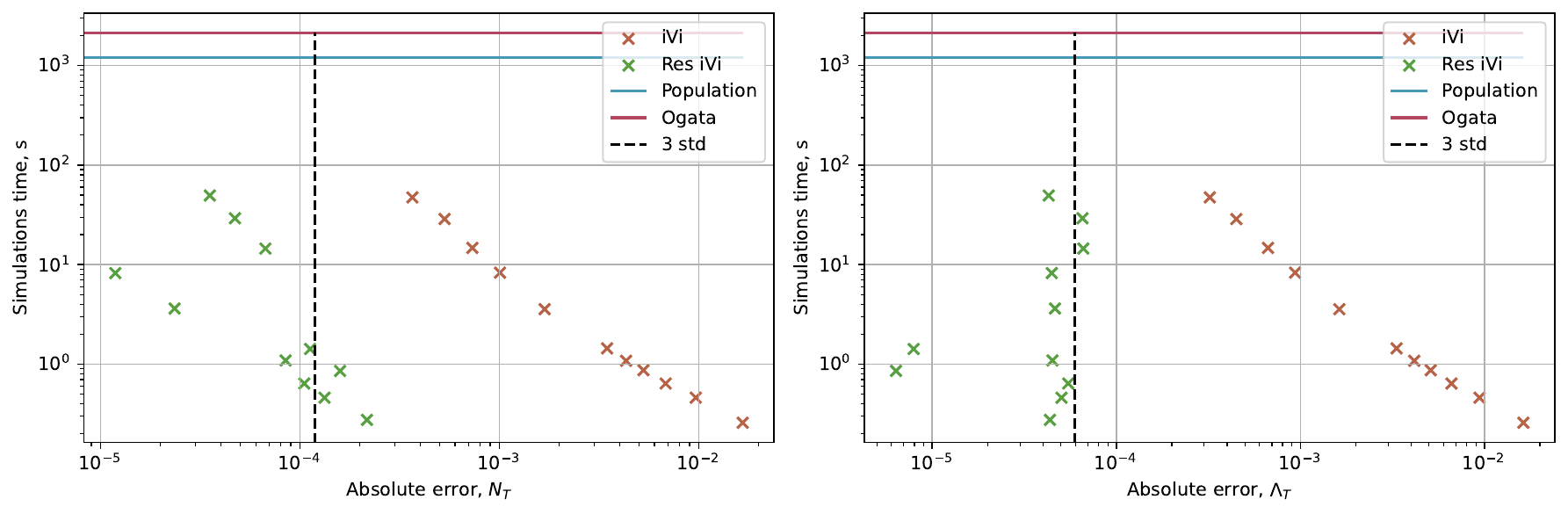}
    \caption{Simulation time for $10^5$ trajectories ($y$-axis) versus the absolute error of the Monte Carlo estimator of the Laplace transform ($x$-axis). Crosses represent estimators obtained with the iVi scheme under different discretization steps. Horizontal bars indicate the exact methods: Population (blue) and Ogata (purple).}
    \label{fig:sim_times_res}
    \end{center}
\end{figure}

\subsection{Markovian version of the iVi scheme}
\label{subsec:markov_scheme}

In this subsection, we discuss a modification of Algorithms~\ref{alg:simulation} and \ref{alg:simulation_res} that can be applied when the Hawkes process is Markovian.  

For an arbitrary kernel $K$, the complexity of the simulation scheme is quadratic due to the computation of $\alpha_i^n$ given by \eqref{eq:alphai} and \eqref{eq:alphai_res}. However, in the exponential case
\[
K_{\mathrm{exp}}^{b, c}(t) = c e^{-b t}, \quad b > 0, \ c > 0, \ t \geq 0,
\]
we have
\begin{equation}\label{eq:exp_recursive_alpha}
    k_{i+1}^n = e^{-b T / n} k_i^n, \quad i = 0, \ldots, n-1,
\end{equation}
so that
\begin{align*}
\alpha_i^n &= \int_{t_i^n}^{t_{i+1}^n} g_0(s)\, ds + \sum_{j=0}^{i-1} k_{i-j}^n \, \widehat N_{j,j+1}^n \\
&= \int_{t_i^n}^{t_{i+1}^n} g_0(s)\, ds + e^{-b T / n} \Big(\alpha_{i-1}^n - \int_{t_{i-1}^n}^{t_i^n} g_0(s)\, ds \Big) + k_1^n \widehat N_{i-1,i}^n, \quad i = 1, \ldots, n-1.
\end{align*}

This formula leads to a scheme with linear complexity, as all information about the past is summarized in the previous value $\alpha_{i-1}^n$. The same argument applies to the resolvent scheme, since the resolvent of the exponential kernel is itself an exponential kernel.

Similarly, a Markovian version of Ogata's Algorithm~\ref{alg:ogata} can be easily derived by keeping track of the value
\[
Y(t) = \sum_{\tau_i \leq t} K_{\mathrm{exp}}^{b, c}(t - \tau_i),
\]
which satisfies $Y(t+h) = e^{-b h} Y(t), \ h > 0$. Moreover, in the Markovian case, there exists an exact algorithm to simulate the next arrival $\tau_{i+1}$ conditionally on $\mathcal{F}_{\tau_i}$, see \citet{Dassios13}.

We illustrate the convergence of the Laplace transform of $N_T$ and $\Lambda_T$ for $\mu = 10$ and $T = 2$, corresponding to the exponential kernel with $b = 5$ and $c = 4$. The results for the exact scheme, the Markovian version of Ogata's algorithm, and the iVi schemes are presented in Figures~\ref{fig:cf_conv_exp} and \ref{fig:sim_times_exp}. Again, the iVi scheme achieves substantial acceleration (about 50 times faster than the exact method for the Resolvent iVi scheme with 200 time steps), even compared with the optimal exact method, thanks to time discretization and vectorization of the scheme.

\begin{figure}[H]
    \begin{center}
    \includegraphics[width=1\linewidth]{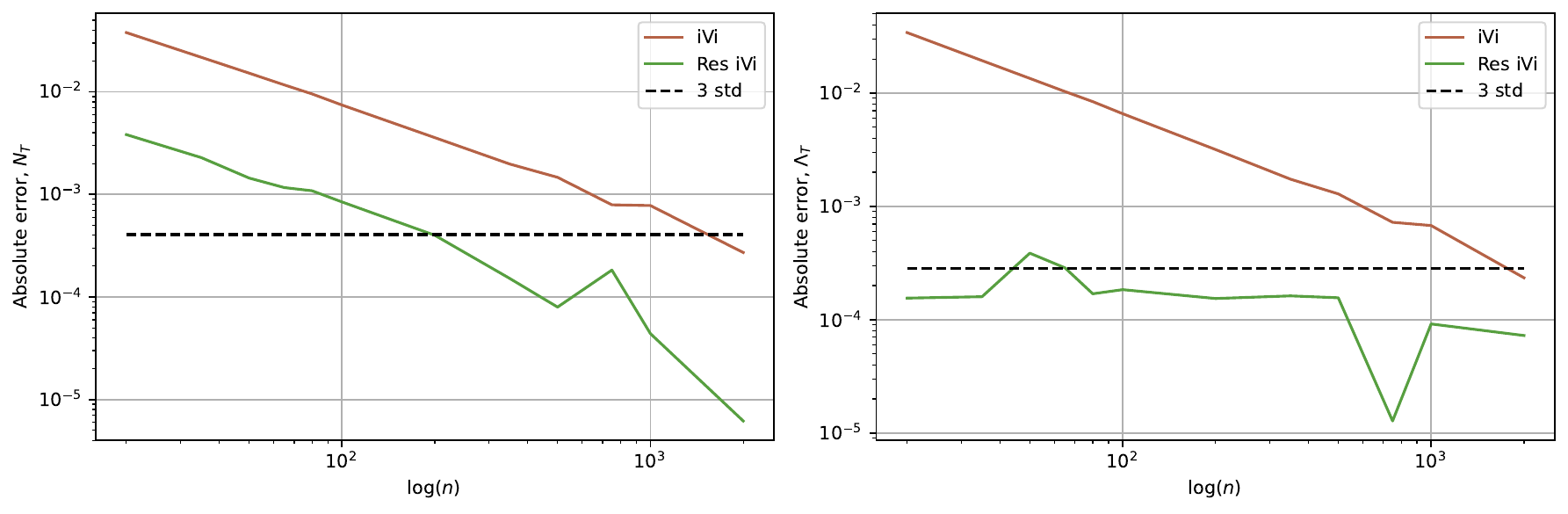}
    \caption{Convergence of the Laplace transform with $w = -\frac{1}{\mathbb{E}[N_T]}$ for $N_T$ (left) and $\Lambda_T$ (right). The $x$-axis shows the number of time steps, and the $y$-axis shows the absolute error of the Monte Carlo estimator.}    \label{fig:cf_conv_exp}
    \end{center}
\end{figure}

\begin{figure}[H]
    \begin{center}
    \includegraphics[width=1\linewidth]{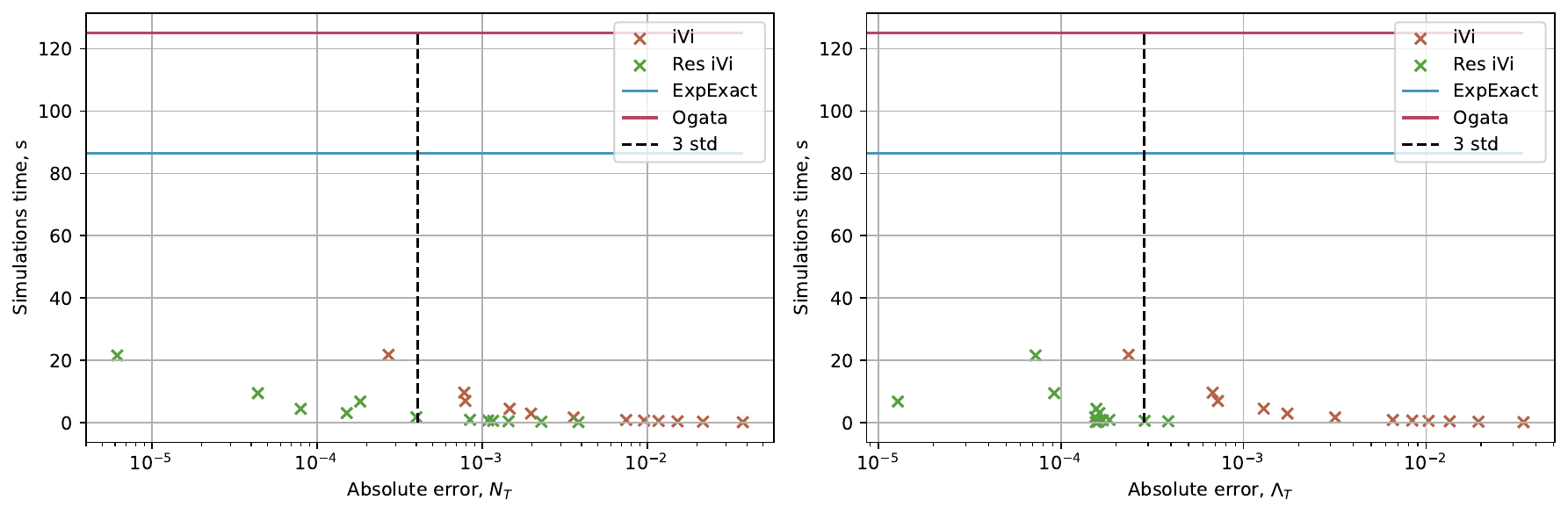}
    \caption{Simulation time for $10^5$ trajectories ($y$-axis) versus the absolute error of the Monte Carlo estimator of the Laplace transform ($x$-axis). Crosses represent estimators obtained with the iVi scheme under different discretization steps. Horizontal bars indicate the exact methods: Population (blue) and Ogata (purple).}
    \label{fig:sim_times_exp}
    \end{center}
\end{figure}

\begin{remark}
A similar reasoning applies when the kernel is given by a sum of exponentials:
\begin{align}\label{eq:sumexpkernel}
K(t) = \sum_{k=1}^m c_k e^{-b_k t} = \sum_{k=1}^m K_{\mathrm{exp}}^{b_k, c_k}(t), \quad b_k > 0, \ c_k > 0, \ t \geq 0.    
\end{align}
In this case, we can decompose $k_i^n$ as
\[
k_i^n = \sum_{k=1}^m k_i^{n,(k)}, \quad i = 0, \ldots, n-1,
\]
where $k_i^{n,(k)}$ corresponds to the kernel $K_{\mathrm{exp}}^{b_k, c_k}$ and satisfies the recursion \eqref{eq:exp_recursive_alpha}.  

The Markovian state is then given by the vector ${y_i} = (y_{i,1}, \ldots, y_{i,m})$, with
\[
y_{i,k} = \sum_{j=0}^{i-1} k_{i-j}^{n,(k)} \, \widehat N_{j,j+1}^n,
\]
which evolves according to
\[
y_{i+1,k} = e^{-b_k T / n} y_{i,k} + k_1^{n,(k)} \widehat N_{i,i+1}^n,
\]
and the increment $\alpha_i^n$ is then given by
\[
\alpha_i^n = \int_{t_i^n}^{t_{i+1}^n} g_0(s)\, ds + \sum_{k=1}^m y_{i,k}.
\]
Sum of exponential kernels as in \eqref{eq:sumexpkernel} can be used to approximate other kernels $K$, giving rise to so-called multifactor Markovian approximations of the Hawkes process with kernel $K$; see, for instance, \cite*[Section 7.2]{abietal2021weak}. This approach yields a  computational complexity of order $\mathcal{O}(mn)$; thus, when the number of factors $m$ is smaller than the number of time steps $n$, it offers a significant computational advantage.
\end{remark}

\section{Proof of weak convergence}
\label{sect:convergence_proof}
This section is dedicated to the proof of convergence of the Hawkes iVi scheme, as stated in Theorem \ref{theorem:convergence}.
\subsection{Preliminary results}
We begin with some basic properties on the moments of the random variables $\left(\alpha\,,\xi\,,\widehat N\,, \widehat \Lambda\right)$ produced by Algorithm \ref{alg:simulation}. We recall that the latter leads to the following structure at step $i \leq n-1\,$:
    \begin{itemize}
        \item The quantity $\alpha^n_i$ is given by \eqref{eq:alphai}.
        \item Conditional on $\alpha^n_i\,$, $\xi^n_i$ follows an Inverse Gaussian distribution with parameters $\left(\frac{\alpha^n_i}{1 - k_0^n}\,, \left(\frac{\alpha^n_i}{k_0^n} \right)^2 \right)\,$.
        \item Conditional on $\xi^n_i\,$, $\widehat N^n_{i,i+1}$ follows a Poisson distribution with parameter $\xi^n_i\,$.
        \item The increment $\widehat \Lambda^n_{i,i+1}$ is defined as $\alpha^n_i + k_0^n\,\widehat N^n_{i,i+1}\,$.
    \end{itemize}
    The filtrations $\left(\mathcal F^n_t\right)_{t \leq T, n \geq 1}$ are defined in \eqref{eq:def_filtration}.
\begin{lemma}    
\label{lemma:estimates_random_variables}
Let $n \geq 1\,$. Then, the random variables $\left(\alpha^n_i\right)_{i \leq n-1}\,$, $\left(\xi^n_i\right)_{i \leq n-1}\,$, $\left(\widehat N^n_{i,i+1}\right)_{i \leq n-1}$ and $\left(\widehat \Lambda^n_{i,i+1}\right)_{i \leq n-1}$ have finite moments of order 2. Moreover, we have
    \begin{equation}
    \label{eq:exp_N_hat}
    \mathbb E \left[\widehat N^n_{i,i+1}\,\middle|\,\mathcal F^n_{t^n_j} \right] =  \mathbb E \left[\widehat \Lambda^n_{i,i+1} \,\middle|\,\mathcal F^n_{t^n_j}\right]\,, \quad j \leq i\leq n-1\,.
    \end{equation}
    \begin{equation}
        \label{eq:exp_Z_hat_squared}
        \mathbb E \left[\left(\widehat N^n_{i,i+1} - \widehat \Lambda^n_{i,i+1}\right)^2\,\middle|\,\mathcal F^n_{t^n_j} \right] = \left(\left(k_0^n\right)^2 + \left(1 - k_0^n\right)^2\right)\, \mathbb E \left[ \widehat \Lambda^n_{i,i+1}\,\middle|\,\mathcal F^n_{t^n_j} \right]\,, \quad j \leq  i \leq n-1\,.
    \end{equation}
\end{lemma}
\begin{proof}
    We prove the existence of second order moments by induction on $i \leq n-1\,$. Firstly, $\widehat \alpha^n_0 = \int_0^{\frac T n}\, g_0(s)\,ds$ is deterministic. Moreover $ \xi^n_0 \sim IG\left(\frac{\alpha^n_0}{1 - k_0^n}\,, \left( \frac{\alpha^n_0}{ k_0^n}\right)^2\right)$ has finite moments of any order (see Appendix \ref{App:IG}). Then, we have
    $$
    \mathbb E \left[\left(\widehat N^n_{0,1}\right)^2 \right] = \mathbb E \left[\mathbb E \left[\left(\widehat N^n_{0,1}\right)^2\,\middle|\,\xi^n_0 \right] \right] = \mathbb E \left[\xi^n_0 + \left(\xi^n_0\right)^2\right] < + \infty \,,
    $$
    since $\widehat N^n_{0,1}$ follows a Poisson distribution with parameter $\xi^n_0\,$, conditional on the latter. Finally, by linear combination,
    $$
    \widehat \Lambda^n_{0,1} = \alpha^n_0 + k_0^n\,\widehat N^n_{0,1}
    $$
    also admits a finite moment of order 2.

    We next fix $i \leq n-1$ and assume that $\left(\alpha^n_j\right)_{j < i}\,$, $\left(\xi^n_j\right)_{j < i}\,$, $\left(\widehat N^n_{j,j+1}\right)_{j < i}$ and $\left(\widehat \Lambda^n_{j,j+1}\right)_{j <i}$ all have finite moments of order $2$. Then it is also the case for 
    $$
    \alpha^n_i = \int_{t^n_i}^{t^n_{i+1}}\,g_0(s)\,ds + \sum_{j = 0}^{i-1}\, k^n_{i-j}\, \widehat N^n_{j,j+1}\,.
    $$
    Using the formula \eqref{eq:IG_moments} for the moments of the Inverse Gaussian distribution, we obtain
    $$
    \mathbb E \left[\left(\xi^n_i\right)^2 \right] = \mathbb E \left[\mathbb E \left[\left(\xi^n_i\right)^2\,\middle|\,\alpha^n_i \right] \right] = \mathbb E \left[\frac{\left(k_0^n\right)^2}{\left(1 - k_0^n\right)^3}\,\alpha^n_i + \left(\frac{\alpha^n_i}{1 - k_0^n} \right)^2 \right] <+ \infty\,.
    $$
    Finally, similarly to the case $i = 0\,$, $\widehat N^n_{i,i+1}$ and $\widehat \Lambda^n_{i,i+1}$ have finite second-order moments.

    For \eqref{eq:exp_N_hat}, we use the tower property of conditional expectation,
    \begin{align}
    \label{eq:conditional_exp_N_alpha}
        \mathbb E \left[\widehat N^n_{i,i+1}\,\middle|\,\mathcal F^n_{t^n_j} \right] &= \mathbb E \left[\mathbb E \left[\widehat N^n_{i,i+1}\,\middle|\,\mathcal F^n_{t^n_j}\,, \xi^n_i \right]\,|\,\mathcal F^n_{t^n_j} \right] \nonumber\\
        &= \mathbb E \left[\xi^n_i\,\middle|\,\mathcal F^n_{t^n_j}\right] \nonumber\\
        &= \mathbb E \left[ \mathbb E \left[ \xi^n_i\,\middle|\,\mathcal F^n_{t^n_j}\,, \alpha^n_i\right]\,\middle|\,\mathcal F^n_{t^n_j}\right] \nonumber\\
        &= \frac{1}{1 - k_0^n}\,\mathbb E \left[\alpha^n_i\,\middle|\,\mathcal F^n_{t^n_j} \right]\,,\quad j \leq i \leq n-1\,.
    \end{align}
    Inserting this into \eqref{eq:hatLambdasample}, we obtain
    \begin{align*}
        \mathbb E \left[\widehat \Lambda^n_{i,i+1}\,\middle |\,\mathcal F^n_{t^n_j} \right] &= \mathbb E \left[\alpha^n_i\,\left(1 + \frac{k_0^n}{1 - k_0^n} \right)\,\middle |\, \mathcal F^n_{t^n_j} \right] \\
        &= \frac{1}{1 - k_0^n}\, \mathbb E \left[\alpha^n_i\,\middle |\,\mathcal F^n_{t^n_j} \right] \\
        &= \mathbb E \left[\widehat N^n_{i,i+1}\,\middle |\,\mathcal F^n_{t^n_j} \right]\,, \quad j \leq i \leq n-1\,,
    \end{align*}
    thus proving \eqref{eq:exp_N_hat}.

We now move to the second equality. Using \eqref{eq:hatLambdasample}, we have
\begin{align*}
    \mathbb E \left[\left(\widehat N^n_{i,i+1} - \widehat \Lambda^n_{i,i+1} \right)^2\,\middle | \, \mathcal{F}^n_{t^n_j} \right] &= \mathbb E \left[ \left(\left(1 - k_0^n\right)\, \widehat N^n_{i,i+1} - \alpha^n_i\right)^2\,\middle |\, \mathcal F^n_{t^n_j}\right]\,, \quad j \leq i \leq n-1\,.
\end{align*}
An application of the tower property of conditional expectation, along with the expression of second-order moments of Poisson and Inverse Gaussian (see \eqref{eq:IG_moments}) variables, yields that
\begin{align*}
\mathbb E \left[\left(\widehat N^n_{i,i+1} - \Lambda^n_{i,i+1} \right)^2\,\middle |\,\mathcal F^n_{t^n_j} \right]    &=\mathbb E \left[\mathbb E \left[\left(1 - k_0^n\right)^2\,\left(\widehat N^n_{i,i+1}\right)^2 - 2\,\left(1 - k_0^n\right)\,\widehat N^n_{i,i+1}\,\alpha^n_i + \left(\alpha^n_i\right)^2\,\middle |\,\mathcal F^n_{t^n_j}\,,\xi^n_i\,,\alpha^n_i\right] \,\middle |\, \mathcal F^n_{t^n_j}\right] \\
    &= \mathbb E \left[\left(1 - k_0^n\right)^2\,\left(\left(\xi^n_i\right)^2 + \xi^n_i\right) - 2\,\left(1 - k_0^n\right)\, \xi^n_i\, \alpha^n_i + \left(\alpha^n_i\right)^2\,\middle |\,\mathcal F^n_{t^n_j} \right]\\
    &= \mathbb E \left[\mathbb E \left[\left(1 - k_0^n\right)^2\,\left(\left(\xi^n_i\right)^2 + \xi^n_i\right) - 2\,\left(1 - k_0^n\right)\, \xi^n_i\, \alpha^n_i + \left(\alpha^n_i\right)^2\,\middle |\,\mathcal F^n_{t^n_j}\,,\alpha^n_i\right]\,\middle |\,\mathcal F^n_{t^n_j} \right]\\
    &= \mathbb E \left[\frac{\left(k_0^n\right)^2}{1 - k_0^n}\,\alpha^n_i + \left(\alpha^n_i\right)^2 + \left(1 - k_0^n\right)\,\alpha^n_i - 2\,\left(\alpha^n_i\right)^2 + \left(\alpha^n_i\right)^2\,\middle |\,\mathcal F^n_{t^n_j} \right] \\
    &= \frac{\left(k_0^n\right)^2 + \left(1 - k_0^n\right)^2}{1 - k_0^n}\,\mathbb E \left[\alpha^n_i\,\middle |\, \mathcal F^n_{t^n_j} \right]\,, \quad j \leq i \leq n-1\,,
\end{align*}
which combined with \eqref{eq:exp_N_hat} and \eqref{eq:conditional_exp_N_alpha} gives \eqref{eq:exp_Z_hat_squared}.
\end{proof}

The following lemma identifies a martingale that will be pivotal in the proofs.
\begin{lemma}
    \label{lemma:martingales}
    For $n \geq 1\,$, the process $\left(Z^n\right)_{t \leq T} := \left(N^n_t - \Lambda^n_t\right)_{t \leq T}$ is a square-integrable $\left(\mathcal F^n_t\right)_{t \leq T}$-martingale. Its quadratic variation is given by
    $$
    \left[ Z^n\right]_t = \sum_{i = 0}^{\lfloor nt/T\rfloor - 1}\,\left( \widehat N^n_{i,i+1} - \widehat \Lambda^n_{i,i+1}\right)^2\,, \quad t \leq T\,.
    $$
    Moreover, $\left([Z^n] - \left[\left(k_0^n\right)^2 + \left(1 - k_0^n\right)^2\right]\, \Lambda^n\right)$ is a martingale, which implies
    \begin{equation}
        \label{eq:exp_quadratic_var_Z}
        \mathbb E \left[\left[Z^n\right]_t \right] = \left(\left(k_0^n\right)^2 + \left(1 - k_0^n\right)^2\right)\, \mathbb E \left[ \Lambda^n_t\right]\,, \quad t \leq T\,.
    \end{equation}
\end{lemma}
\begin{proof}
    We recall that square-integrability of all the quantities is provided by Lemma \ref{lemma:estimates_random_variables}. Let $s \leq t \leq T\,$, we define $i := \lfloor nt/T\rfloor$ and $j := \lfloor ns/T \rfloor\,$. We have 
    \begin{align*}
        \mathbb E \left[N^n_t \,\middle|\, \mathcal F^n_s\right] &= \mathbb E \left[\sum_{l = 0}^{i - 1}\, \widehat N^n_{l,l+1}\, \middle|\, \mathcal F^n_{t^n_j} \right] \\
        &= N^n_s + \sum_{l = j}^{i-1}\, \mathbb E \left[ \widehat N^n_{l,l+1}\,\middle|\, \mathcal F^n_{t^n_j}\right] \\
        &= N^n_s + \sum_{l = j}^{i-1}\, \mathbb E \left[ \widehat \Lambda^n_{l,l+1}\,\middle|\, \mathcal F^n_{t^n_j}\right] \\
        &= N^n_s + \mathbb{E}\left[\Lambda^n_t - \Lambda^n_s\,\middle|\, \mathcal F^n_s\right] \,,
    \end{align*}
where we used \eqref{eq:exp_N_hat} to compute the conditional expectations.
This proves the martingality of $Z^n = N^n - \Lambda^n\,$. To identify its quadratic variation, we notice that it has bounded variation on $[0\,,T]$ (as it jumps a finite number of times) and is therefore a quadratic pure-jump martingale, see \citet[Theorem~II.26]{protter2005stochastic}. Since it also starts from 0, we deduce that its quadratic variation is the sum of the squares of its jumps, i.e.
$$
\left[Z^n\right]_t = \left[N^n - \Lambda^n  \right]_t = \sum_{i = 0}^{\lfloor nt/T \rfloor - 1}\, \left(\widehat N^n_{i,i+1} -  \widehat \Lambda^n_{i,i+1}\right)^2\,, \quad t \leq T\,.
$$
By taking the expectation, applying \eqref{eq:exp_Z_hat_squared} and recalling the definition of $\Lambda^n$ in \eqref{eq:def_process_Lambda}, we obtain 
\begin{align*}
    \mathbb E \left[[Z^n]_t - [Z^n]_s\,\middle |\, \mathcal F^n_s \right] &= \sum_{l = j}^{i - 1}\, \mathbb E \left[\left(\widehat N^n_{l,l+1} - \widehat \Lambda^n_{l,l+1}\right)^2\,\middle |\,\mathcal F^n_{t^n_j} \right] \\
    &= \left(\left(k_0^n\right)^2 + \left(1 - k_0^n\right)^2\right)\, \mathbb E \left[\Lambda^n_t - \Lambda^n_s\,\middle |\, \mathcal F^n_s\right] \,,
\end{align*}
proving that $\left([Z^n] - \left[\left(k_0^n\right)^2 + \left(1 - k_0^n\right)^2\right]\, \Lambda^n\right)$ is a martingale.
\end{proof}

A very convenient feature of our scheme is that it can be rewritten as a stochastic Volterra equation with a measure-valued kernel, which will be useful for deriving estimates on our processes. This is formulated in the following lemma.
\begin{lemma}
    \label{lemma:Volterra}
    For $n \geq 1\,$, define the nonnegative discrete measure $K^n := \sum_{i = 0}^{n-1}\, k^n_{i}\, \delta_{t^n_i}$ on $\left([0\,,T]\,, \mathcal B_{[0\,,T]}\right)\,$. Then, we have 
    \begin{equation}
    \label{eq:volterra_equation}
    \Lambda^n_t = \int_0^{\lfloor \frac{nt}{T}\rfloor \frac{T}{n}}\,g_0(s)\,ds + \int_{[0,t]}\, N^n_{t-s}\,K^n(ds)\,, \quad t \leq T\,.
    \end{equation}
\end{lemma}
\begin{proof}
We recall that combining \eqref{eq:alphai} and \eqref{eq:hatLambdasample} leads to 
$$
\widehat \Lambda^n_{i,i+1} = \int_{t^n_i}^{t^n_{i+1}}\,g_0(s)\,ds + \sum_{j = 0}^i\,k^n_{i-j}\,\widehat N^n_{j,j+1}\,, \quad i \leq n-1\,, \quad n \geq 1\,.
$$
    Let $t \leq T$ and set $i := \lfloor nt /T \rfloor\,$, then 
    \begin{align*}
        \Lambda^n_t &= \sum_{j = 0}^{i - 1}\, \widehat \Lambda^n_{j,j+1} \\
        &= \int_0^{t^n_i}\,g_0(s)\,ds + \sum_{j = 0}^{i-1}\, \sum_{l = 0}^{j}\,k^n_{j -l}\,\widehat N^n_{l,l+1} \\
        &= \int_0^{t^n_i}\,g_0(s)\,ds + \sum_{j = 0}^{i-1}\, \sum_{l = 0}^{j}\,k^n_l \,\widehat N^n_{j-l, j-l+1} \\
        &= \int_0^{t^n_i}\,g_0(s)\,ds + \sum_{l = 0}^{i-1}\,k^n_l\,\sum_{j = l}^{i-1}\, \widehat N^n_{j-l, j-l+1} \\
        &= \int_0^{t^n_i}\,g_0(s)\,ds + \sum_{l = 0}^{i-1}\,k^n_l \, N^n_{t^n_{i-l}} \,.
    \end{align*}
From the definition of $N^n$ in \eqref{eq:def_process_N}, and the fact that our partition of $[0\,,T]$ is uniform, we have $N^n_{t^n_{i-l}} = N^n_{t - t^n_l}\,$. Since $N^n_0 = 0\,$, we deduce that 
$$
\sum_{l = 0}^{i-1}\,k^n_{l}\, N^n_{t^n_{i-l}} = \int_{[0,t]}\,N^n_{t-s}\,K^n(ds)\,,
$$
which concludes the proof of \eqref{eq:volterra_equation}.
\end{proof}

\subsection{Estimates}
We are now ready to state the estimates on the process $\Lambda^n$ that will be central in proving the tightness of our processes. We make use of the notation $G_0 = \int_0^{\cdot}\,g_0(s)\,ds\,$.
\begin{lemma}
\label{lemma:estimates_process}
We have the following moment bounds,
\begin{equation}
\label{eq:bound_exp_Lambda}
C_{\Lambda} := \sup_{n \geq 1}\,\mathbb E \left[\sup_{0 \leq t \leq T} \Lambda^n_t\right] < + \infty\,,
\end{equation}
\begin{equation}
    \label{eq:bound_exp_Lambda_2}
    C_{\Lambda^2} := \sup_{n \geq 1}\, \mathbb E \left[ \sup_{0 \leq t \leq T}\, \left(\Lambda^n_t\right)^2\right] < + \infty\,.
\end{equation}
Moreover, for any $j < i \leq n\,$,
\begin{equation}
\label{eq:bound_variation_alpha}
\Lambda^n_{t^n_i} - \Lambda^n_{t^n_j} \leq G_0(t^n_i) - G_0(t^n_j) + \left\|N^n\right\|_{\infty}\, \left(\int_{0}^{t^n_{j+1}}\, \left|K(s + t^n_i - t^n_j) - K(s)\right|\,ds + \int_{0}^{t^n_i - t^n_j}\, K(s)\,ds\right)\,.
\end{equation}
\end{lemma}
\begin{proof}
    We first prove \eqref{eq:bound_exp_Lambda_2}, which implies \eqref{eq:bound_exp_Lambda} using $|x| \leq 1 + x^2\,$. Fix $t \leq T$ and $n \geq 1\,$, we recall that thanks to the Volterra equation \eqref{eq:volterra_equation},
    $$
    \Lambda_t^n  \leq G_0(T) + \int_{[0,t]}\,N^n_{t-s}\,K^n(ds)\,.
    $$
    Taking squares and applying Jensen's inequality on the measure $\frac{K^n(ds)}{K^n([0,t])}\,$, we have 
    $$
    \left(\Lambda^n_t\right)^2 \leq 2\,G_0^2(T) + 2\,\int_0^T\,K(s)\,ds \,\int_{[0,t]}\,\left(N^n_{t-s}\right)^2\,K^n(ds) \,,
    $$
    where we used the fact that $K^n([0\,,t]) \leq K^n([0\,,T]) = \int_0^T\,K(s)\,ds\,$.
    We can then leverage the martingale property of $Z^n = N^n - \Lambda^n$ from Lemma \ref{lemma:martingales} to apply the BDG inequality. We obtain
    \begin{align*}
        \mathbb E \left[\left(N^n_t\right)^2 \right] &\leq 2\, \mathbb E \left[\left(Z^n_t\right)^2 \right] + 2 \,\mathbb E \left[\left(\Lambda^n_t\right)^2 \right] \\
        &\leq 8\, \mathbb E \left[\left[Z^n\right]_t \right] + 2 \,\mathbb E \left[\left(\Lambda^n_t\right)^2 \right]\\ 
        &\leq 8\,\left(\left(k_0^n\right)^2 + \left(1 - k_0^n\right)^2\right)\,\mathbb E \left[\Lambda^n_t \right] + 2 \,\mathbb E \left[\left(\Lambda^n_t\right)^2 \right]\\
        &\leq 8\,\left(\left(k_0^n\right)^2 + \left(1 - k_0^n\right)^2\right)\,\mathbb E \left[1 + \left(\Lambda_t^n\right)^2 \right] + 2 \,\mathbb E \left[\left(\Lambda^n_t\right)^2 \right]\,.
    \end{align*}
    This, along with the fact that $k_0^n \to 0$ as $n \to \infty\,$, proves that we can find a constant $C > 0\,$, independent of $t$ and $n\,$, such that
    $$
    \mathbb E \left[\left(\Lambda^n_t\right)^2 \right] \leq C + C\,\int_{[0,t]}\,\mathbb E \left[\left(\Lambda^n_{t-s}\right)^2 \right]\,K^n(ds)\,.
    $$
    We now choose $\theta > 0$ such that 
    $$
    C \,\int_0^T\,e^{-\theta s}\,K(s)\,ds \leq \frac{1}{2}\,.
    $$
    Such a number exists by the dominated convergence theorem. It allows us to bound
    \begin{align*}
        \int_{[0,T]}\,e^{- \theta s}\,K^n(ds) &= \sum_{i = 0}^{n-1}\,e^{- \theta t^n_i}\, \int_{t^n_i}^{t^n_{i+1}}\,K(s)\,ds \\
        &= \sum_{i = 0}^{n-1}\,\int_{t^n_i}^{t^n_{i+1}}\,e^{\theta (s - t^n_i)}\,e^{- \theta s}\,K(s)\,ds \\
        &\leq e^{\theta T / n}\, \int_0^T\,e^{- \theta s}\,K(s)\,ds\\
        & \leq \frac{1}{2C}\,e^{\theta T / n}\,, \quad n \geq 1\,.
    \end{align*}
    Thus, there exists a rank $n_1\,$, such that
    $$
    \int_{[0,T]}\,e^{- \theta s}\,K^n(ds) \leq \frac{3}{4C}\,, \quad n \geq n_1\,.
    $$
    Hence, we obtain
    \begin{align*}
        \sup_{0 \leq t \leq T}\, e^{- \theta t}\, \mathbb E \left[\left(\Lambda^n_t\right)^2 \right] &\leq C + C \sup_{0 \leq t \leq T}\, \int_{[0,t]}\,e^{- \theta(t-s)}\, \mathbb E \left[\left(\Lambda^n_{t-s}\right)^2\right]\,e^{- \theta s}\,K^n(ds) \\
        &\leq C + C\,\left(\sup_{0 \leq t \leq T}\,e^{- \theta t}\, \mathbb E \left[\left(\Lambda^n_{t} \right)^2\right] \right)\,\int_{[0,T]}\,e^{- \theta s}\,K^n(ds) \\
        &\leq C + \frac{3}{4}\, \sup_{0 \leq t \leq T}\,e^{- \theta t}\, \mathbb E \left[\left(\Lambda^n_{t} \right)^2\right]\,, \quad n \geq n_1\,.
    \end{align*}
    For each $n\,$, the quantity $\sup\limits_{0 \leq t \leq T}\,e^{- \theta t}\, \mathbb E \left[ \left(\Lambda^n_t\right)^2\right]$ is finite, as it is bounded by $\mathbb E \left[\left(\Lambda^n_T \right)^2 \right]$ which is finite thanks to Lemma \ref{lemma:estimates_random_variables}. Therefore, we obtain,
    $$
    \sup_{0 \leq t \leq T}\,e^{- \theta t}\,\mathbb E \left[\left(\Lambda^n_t \right)^2\right] \leq 4\,C\,, \quad n \geq n_1\,,
    $$
    and finally,
    $$
    \mathbb E \left[(\Lambda^n_T)^2 \right] \leq e^{\theta T}\,\sup_{0 \leq t \leq T}\,e^{-\theta t}\, \mathbb E \left[ \left(\Lambda^n_t\right)^2\right] \leq 4\,C\,e^{\theta T}\,, \quad n \geq n_1\,.
    $$
    Since $\Lambda^n$ is non-decreasing, this proves \eqref{eq:bound_exp_Lambda_2}.

    For \eqref{eq:bound_variation_alpha}, we fix $n \geq 1$ and $j < i \leq n\,$. We have 
    \begin{align*}
        \Lambda^n_{t^n_i} - \Lambda^n_{t^n_j} &= \int_{t^n_j}^{t^n_i}\,g_0(s)\,ds + \sum_{l = 0}^i\,N^n_{t^n_i - t^n_l}\,k^n_l - \sum_{l = 0}^j\,N^n_{t^n_j - t^n_l}\,k^n_l \\
        &= \int_{t^n_j}^{t^n_i}\,g_0(s)\,ds + \sum_{l = 0}^{i -j -1}\,N^n_{t^n_i - t^n_l}\,k^n_l + \sum_{l = i-j}^i\,N^n_{t^n_i - t^n_l}\,k^n_l - \sum_{l = 0}^j \,N^n_{t^n_j - t^n_l}\,k^n_l\\
        &= \int_{t^n_j}^{t^n_i}\,g_0(s)\,ds + \sum_{l = 0}^{i - j - 1}\,N^n_{t^n_i - t^n_l}\,k^n_l + \sum_{l = 0}^j\,N^n_{t^n_j - t^n_l}\,\left(k^n_{l + i -j} - k^n_l\right) \\
        &\leq \int_{t^n_j}^{t^n_i}\,g_0(s)\,ds + \|N^n\|_{\infty}\, \left(\sum_{l = 0}^{i-j-1}\,k^n_l + \sum_{l = 0}^j\,\left| k^n_{l + i - j} - k^n_l\right|\right)\,.
    \end{align*}
    We conclude by computing
    $$
    \sum_{l = 0}^{i - j-1}\,k^n_l = \int_0^{t^n_i - t^n_j}\,K(s)\,ds\,,
    $$
    and
    \begin{align*}
        \sum_{l = 0}^j\,\left| k^n_{l + i - j} - k^n_l\right| = \sum_{l = 0}^{j}\left|\int_{t^n_l}^{t^n_{l+1}}\,\left(K(s + t^n_i - t^n_j) - K(s) \right)\,ds \right|
        \leq \int_0^{t^n_{j+1}}\, \left|K(s + t^n_i - t^n_j) - K(s) \right|\,ds\,.
    \end{align*}
\end{proof}

\begin{remark}
\label{remark:additional_bound}
For any $n \geq 1\,$, $Z^n = N^n - \Lambda^n$ is a martingale, therefore by the BDG inequality,
$$
\mathbb E \left[\sup_{0 \leq t \leq T}\,\left|Z^n_t\right|^2\right] \leq 4\,\mathbb E \left[\left[Z^n\right]_T\right] \leq 4\,\left(\left(k_0^n\right)^2 + \left(1-k_0^n\right)^2\right)\, \mathbb E \left[\Lambda^n_T\right]\,,
$$
where we used \eqref{eq:exp_quadratic_var_Z}. Combining this with the bound \eqref{eq:bound_exp_Lambda}, we obtain
\begin{equation}
\label{eq:bound_exp_Z}
    C_Z := \sup_{n \geq 1}\, \mathbb E \left[ \sup_{0 \leq t \leq T}\, \left|Z^n_t\right|^2\right] < + \infty\,,
\end{equation}
since $k_0^n$ goes to $0$ in the limit $n \to \infty\,$.
\end{remark}

We will need one additional property to characterize the limiting processes. Namely, we want to ensure the limit of $\left(N^n\right)_{n \geq 1}$ is a counting process, and therefore only has unit jumps.
\begin{lemma}  \label{lemma:bound_proba_jumps}
We have
$$
\mathbb P \left(\sup_{0 \leq t \leq T}\,\Delta N^n_t > 1 \right) \overset{n \to \infty}{\longrightarrow} 0\,.
$$
\end{lemma}
\begin{proof}
    Let $n \geq 1\,$, we have 
    $$
    \mathbb P \left(\sup_{ 0 \leq t \leq T}\, \Delta\,N^n_t > 1 \right) = \mathbb P \left(\bigcup_{i = 0}^{n-1}\, \left\{\widehat N^n_{i,i+1} \geq 2 \right\} \right) \leq \sum_{i = 0}^{n-1}\, \mathbb P \left(\widehat N^n_{i,i+1} \geq 2 \right)
    $$
    Then, since $\widehat N^n_{i,i+1} \sim \mathcal P \left( \xi^n_i \right)$ conditional on $\xi^n_i\,$, we have 
    $$
    \mathbb P \left(\widehat N^n_{i,i+1} \geq 2 \right) = \mathbb E \left[1 - e^{- \xi^n_i} -  \xi^n_i\, e^{- \xi^n_i} \right]\,, \quad i \leq n-1\,.
    $$
    Using the bound $1 - e^{-x} - x\,e^{-x} \leq x^2$ for $x \geq 0\,$, we obtain
    \begin{align*}
        \mathbb P \left(\widehat N^n_{i,i+1} \geq 2\right) &\leq \mathbb E \left[\left(\xi^n_i\right)^2 \right] = \mathbb E \left[\mathbb E \left[ \left(\xi^n_i\right)^2\,\middle | \alpha^n_i\right] \right] = \frac{\left(k_0^n\right)^2}{\left(1 - k_0^n\right)^3}\, \mathbb E \left[ \alpha^n_i\right] + \mathbb E \left[\left(\alpha^n_i\right)^2\right]\,, \quad i \leq n-1\,,
    \end{align*}
    where we used \eqref{eq:IG_moments} for computing the second order moment of an Inverse Gaussian variable. We now use the fact that $\alpha^n_i$ is nonnegative (see \eqref{eq:alphai}) and $\alpha^n_i \leq \widehat \Lambda^n_{i,i+1}$ thanks to \eqref{eq:hatLambdasample}. Therefore,
    \begin{align*}
    \mathbb P \left(\widehat N^n_{i,i+1} \geq 2 \right) &\leq \frac{\left(k_0^n\right)^2}{\left(1 - k_0^n\right)^3}\, \mathbb E \left[\widehat \Lambda^n_{i,i+1} \right] + \mathbb E \left[\left(\widehat \Lambda^n_{i,i+1}\right)^2 \right] \\
    &\leq \frac{\left(k_0^n\right)^2}{\left(1 - k_0^n\right)^3}\, \mathbb E \left[\widehat \Lambda^n_{i,i+1} \right] + \mathbb E \left[\widehat \Lambda^n_{i,i+1}\,\max_{0 \leq j \leq n-1}\, \widehat \Lambda^n_{j,j+1} \right]\,, \quad i \leq n-1\,.
    \end{align*}
    Summing this inequality, we get
    $$
    \mathbb P \left(\sup_{0 \leq t \leq T}\,\Delta N^n_t > 1 \right) \leq \mathbb E \left[\Lambda^n_T\, \left(\frac{\left(k_0^n\right)^2}{\left(1 - k_0^n\right)^3} + \max_{0 \leq j \leq n-1}\, \widehat \Lambda^n_{j,j+1} \right) \right] \,.
    $$
    We treat the two terms on the right-hand side separately. First, 
    $$
    \frac{\left(k_0^n\right)^2}{\left(1 - k_0^n\right)^3}\, \mathbb E \left[\Lambda^n_T \right] \leq \frac{\left(k_0^n\right)^2}{\left(1 - k_0^n\right)^3}\, C_{\Lambda} \overset{n \to \infty}{\longrightarrow} 0\,,
    $$
    from \eqref{eq:bound_exp_Lambda} and the fact that $k_0^n \to 0$ in the limit. Secondly,
    \begin{align*}
        \mathbb E \left[\Lambda^n_T\, \max_{0 \leq j \leq n-1}\,\widehat \Lambda^n_{j,j+1}\right] &\leq w\left(G_0,\frac{T}{n}\right)\,\mathbb E \left[\Lambda^n_T\right] + \mathbb E \left[\Lambda^n_T\, \|N^n\|_{\infty}\right]\,\left(\int_0^{T-\frac{T}{n}}\,\left|K\left(t + \frac{T}{n}\right) - K(t)\right| \,dt + k_0^n\right)
    \end{align*}
    from \eqref{eq:bound_variation_alpha}. Here, $w$ stands for the classical modulus of continuity. The right-hand side vanishes in the limit since 
    $$
    w\left(G_0, \frac{T}{n}\right)\,, \quad k_0^n\,, \quad \text{and}\quad  \int_{0}^{T - \frac{T}{n}}\,\left|K\left(t+\frac{T}{n}\right) - K(t)\right|\,dt
    $$
    all converge to $0$ (see Lemma \ref{lemma:translation_continuous_L1}), and
    $$
    \sup_{n \geq 1}\,\mathbb E \left[ \Lambda^n_T\, \|N^n\|_{\infty}\right] < + \infty
    $$
    by Lemma \ref{lemma:estimates_process} and Remark \ref{remark:additional_bound}.
    Thus, we can conclude that
    $$
    \mathbb P \left(\sup_{0 \leq t \leq T}\, \Delta N^n_t > 1 \right) \overset{n \to \infty}{\longrightarrow} 0\,.
    $$
\end{proof}

\subsection{Tightness and characterization of the limiting processes}

We now have all the elements to prove the tightness of our processes in the Skorokhod topology. By \textit{C-tightness}, we mean tightness with almost surely continuous limits, see \citet[Definition~VI.3.25]{jacod2013limit}.
\begin{lemma}
    \label{lemma:tightness}
    The sequence $\left(\Lambda^n\right)_{n \geq 1}$ is C-tight, and $\left( N^n \right)_{n \geq 1}$ is tight.
\end{lemma}
\begin{proof}
    We start by proving C-tightness of $\left(\Lambda^n\right)_{n \geq 1}$ using criterion \textit{(ii)} from \citet[Proposition~VI.3.26]{jacod2013limit}. Firstly, \eqref{eq:bound_exp_Lambda} implies
    $$
    \lim_{R \to \infty}\, \limsup_{n \to \infty}\,\mathbb P \left(\sup_{0 \leq t \leq T}\, \left|\Lambda^n_t\right| \geq R \right) = 0\,. 
    $$
    Then, using \eqref{eq:bound_variation_alpha}, and noticing that we can rewrite the modulus of continuity of $\Lambda^n$ as 
    $$
    w\left(\Lambda^n, \delta\right) = \max_{\substack{0 \leq j \leq i \leq n \\ t^n_i - t^n_j \leq \delta}} \left|\Lambda^n_{t^n_i} - \Lambda^n_{t^n_j}\right|\,, \quad \delta > 0\,,
    $$
    we get 
    $$
w\left(\Lambda^n,\delta\right) \leq w\left(G_0, \delta\right) + \left\|N^n\right\|_{\infty}\, \left(\sup_{0 \leq \eta \leq \delta}\, \int_0^{T-\eta}\, \left|K(s + \eta) - K(s)\right|\,ds + \int_0^{\delta}\, K(s)\,ds \right) \,.
    $$
Since $g_0\,,K \in L^1([0\,, T]\,, \R_+)\,$, the dominated convergence theorem shows that $w(G_0,\delta)$ and $\int_0^{\delta}\,K(s)\,ds$ go to $0$ as $\delta \to 0^+\,$, and Lemma \ref{lemma:translation_continuous_L1} proves that it is also the case for 
$$
\sup_{0 \leq \eta \leq \delta}\, \int_0^{T-\eta} \, \left|K(s + \eta) - K(s)\right|\, ds\,.
$$
Using the bounds \eqref{eq:bound_exp_Lambda}-\eqref{eq:bound_exp_Z}, we get
$$
\sup_{n \geq 1} \mathbb E \left[\left\|N^n\right\|_{\infty}\right] < + \infty\,,
$$
so that finally
\begin{equation}
\label{eq:limit_exp_continuity_modulus_Lambda}
\lim_{\delta \to 0^+}\, \limsup_{n \to \infty}\, \mathbb E \left[w\left(\Lambda^n, \delta\right)\right] = 0\,.
\end{equation}
Thus, $\left(\Lambda^n\right)_{n \geq 1}$ is C-tight.

Finally, we prove that $\left(N^n\right)_{n \geq 1}$ is tight. We use Aldous' criterion on $\left(Z^n\right)_{n \geq 1}\,$, see \citet[Theorem~VI.4.5]{jacod2013limit}. Let $n \geq 1\,$, $\delta > 0\,$, and $0 \leq S_1^n \leq S_2^n \leq T$ two $\left(\mathcal F^n_t\right)_{t \leq T}$-stopping times, such that $S_2^n - S_1^n \leq \delta$ almost surely. Then, using the martingality of $[Z^n] - \left(\left(k_0^n\right)^2 + \left(1 - k_0^n\right)^2\right)\,\Lambda^n$ from Lemma \ref{lemma:martingales}, we have 
\begin{align*}
\mathbb E \left[\left(Z^n_{S_2^n} - Z^n_{S_1^n} \right)^2\right] &= \mathbb E \left[[Z^n]_{S_2^n} - [Z^n]_{S_1^n} \right]\\
&= \left(\left(k_0^n\right)^2 + \left(1 - k_0^n\right)^2\right)\,\mathbb E \left[\Lambda^n_{S_2^n} - \Lambda^n_{S_1^n}\right] \\
&\leq \left( \left(k_0^n\right)^2 + \left(1 - k_0^n\right)^2\right)\, \mathbb E \left[w\left(\Lambda^n,\delta\right)\right]\,.
\end{align*}
Therefore, defining the sets
$$
\mathcal T^n_{\delta} := \left\{S_1, S_2 \,\,\text{$\left(\mathcal F^n_t\right)_{t \leq T}$-stopping times}\,:\, 0 \leq S_1 \leq S_2 \leq \min(T\,, S_1 + \delta) \right\}\,, \quad \delta > 0\,, \quad n \geq 1\,,
$$
we have 
$$
\limsup_{n \geq 1}\,\sup_{\left(S_1^n,S_2^n\right) \in \mathcal T^n_{\delta}}\, \mathbb P \left(\left|Z^n_{S_2^n} - Z^n_{S_1^n}\right| > \varepsilon \right) \leq \limsup_{n \geq 1}\,\frac{\left(k_0^n\right)^2+ \left(1 - k_0^n\right)^2}{\varepsilon^2}\, \mathbb E \left[w\left(\Lambda^n, \delta\right)\right]\,, \quad \delta > 0\,, \quad \varepsilon > 0\,,
$$
which vanishes in the limit $\delta \to 0^+\,$, from \eqref{eq:limit_exp_continuity_modulus_Lambda}. Thus, $\left(Z^n\right)_{n \geq 1}$ is tight for the Skorokhod topology. We conclude by using the fact that $\left(\Lambda^n\right)_{n \geq 1}$ is C-tight, so that $\left(N^n\right)_{n \geq 1} = \left(Z^n + \Lambda^n\right)_{n \geq 1}$ is tight, thanks to \citet[Corollary~VI.3.33]{jacod2013limit}.
\end{proof}

From Lemma \ref{lemma:tightness}, $\left(\Lambda^n\,, N^n\right)_{n \geq 1}$ is tight. By Prokhorov's theorem, it admits a weakly convergent subsequence. The following lemma characterizes the accumulation points.
\begin{lemma}
\label{lemma:characterization_limit}
Let $\left(\Lambda, N\right)$ be a weak accumulation point of $\left(\Lambda^n\,, N^n\right)_{n \geq 1}\,$. We have the following:
\begin{enumerate}[label=(\roman*)]
    \item $\Lambda$ is a continuous, non-decreasing and nonnegative process, starting from $0\,$.
    \item $N$ is a piecewise constant càdlàg process with unit jumps, starting from $0\,$.
    \item $\Lambda = \int_0^{\cdot}\,\lambda_s\,ds$ with $\lambda$ satisfying
    $$
    \lambda_t = g_0(t) + \int_0^t\,K(t-s)\,dN_s\,, \quad t \leq T\,.
    $$
    \item $N - \Lambda$ is a martingale with respect to the filtration generated by $N\,$.
\end{enumerate}
\end{lemma}
\begin{proof}
    We first prove \textit{(i)}, \textit{(ii)} by standard arguments.

    The first point comes directly from the definition of $C$-tightness and the fact that each $\Lambda^n$ is non-decreasing, nonnegative, and starts from $0\,$.

    For \textit{(ii)}, $N_0 = 0$ almost surely as it is also the case for each $N^n\,$. Furthermore, we know that $\mathbb P \left(N^n_t \in \mathbb N\,,\,t \leq T\,, n \geq 1\right) = 1\,$. Since $\left\{f \in D([0\,,T])\,:\, f(t) \in \mathbb N\,,\, t \leq T \right\}$ is a closed set for the Skorokhod topology. We deduce by Portmanteau's theorem that $N$ takes its values in $\N\,$, almost surely. It remains to prove that its jumps can only be equal to $1\,$. From Lemma \ref{lemma:bound_proba_jumps}, we have
    $$
    \mathbb P \left(\sup_{0 \leq t \leq T}\, \Delta N^n_t > 1 \right) \overset{n \to \infty}{\longrightarrow} 0\,.
    $$
    Denoting by $\mathbb P_n'$ the law of $N^n\,$, this can be rewritten as 
    $$
    \mathbb P_n'\left(\mathcal S_{\Delta}^{-1}((1\,, + \infty)) \right) \overset{n \to \infty}{\longrightarrow} 0\,,
    $$
    where $\mathcal S_{\Delta} : f \mapsto \sup\limits_{0 \leq t \leq T}\,\left|\Delta f(t)\right|$ is shown to be continuous for the Skorokhod topology in Lemma \ref{lemma:functionals_continuity}, so that $\mathcal S_{\Delta}^{-1}((1\,, + \infty))$ is an open set, and by applying Portmanteau's theorem, we obtain
    $$
    \Delta N_t \leq 1\,, \quad t \leq T\,, \quad a.s.
    $$
    which concludes the proof of \textit{(ii)}.

    For the proof of \textit{(iii)}, we make use of Skorokhod's representation theorem. Since we have a subsequence $\left(\Lambda^{n_k}\,, N^{n_k}\right)_{k \geq 0}$ converging weakly to $(\Lambda\,, N)\,$, there exists càdlàg processes $\left(\tilde \Lambda^k\,, \tilde N^k\right)_{k \geq 0}$ and $\left(\tilde \Lambda\,, \tilde N\right)$ all defined on the same probability space $\left(\tilde \Omega\,, \tilde{\mathcal{F}}\,, \tilde{\mathbb{P}}\right)\,$, such that:
    \begin{itemize}
        \item For any $k \geq 0\,$, $\left(\tilde \Lambda^k \,, \tilde N^k\right) \sim \left(\Lambda^{n_k}\,, N^{n_k}\right)\,$, and $\left(\tilde \Lambda\,, \tilde N\right) \sim \left(\Lambda\,, N\right)\,$.
        \item For any $\omega \in \tilde \Omega\,$, $\left(\tilde \Lambda^k(\omega)\,, \tilde N^k(\omega)\right) \overset{k \to \infty}{\longrightarrow} \left( \tilde \Lambda(\omega)\,, \tilde N(\omega)\right)$ in the Skorokhod topology.
    \end{itemize}
    Using the same reasoning as above, $\tilde \Lambda$ is almost surely continuous. We use the notation $G_0^n(t) := \int_0^{\lfloor n \frac{t}{T} \rfloor \frac{T}{t}}\,g_0(s)\,ds\,$, and we define the following subsets of $\left(D([0\,,T]\right)^2\,$,
    $$
    \mathcal V^n := \left\{(f,g) \in (D([0\,,T])^2\,:\, f(t) = G_0^n(t) + \int_{[0,t]}\,g(t-s)\,K^n(ds)\,,\,\,t \leq T  \right\}\,, \quad n \geq 1\,,
    $$
    $$
    \mathcal V := \left\{(f,g) \in (D([0\,,T])^2\,:\,f(t) = G_0(t) + \int_0^t\,K(t-s)\,g(s)\,ds\,,\,\,t\leq T \right\}\,.
    $$
    From the Volterra equation of Lemma \ref{lemma:Volterra}, we know that
    $$
    \mathbb P \left(\left(\Lambda^n\,,N^n\right) \in \mathcal V^n \right) = 1\,, \quad n \geq 1\,.
    $$
    By equality in law of the couples,
    $$
    \tilde{\mathbb P}\left(\left(\tilde \Lambda^k\,, \tilde N^k \right) \in \mathcal V^{n_k}\right) = 1\,, \quad k \geq 0\,,
    $$
    leading to 
    $$
    \tilde{\mathbb P}\left(\left(\tilde \Lambda^k\,, \tilde N^k\right) \in \mathcal V^{n_k}\,, \,\, k \geq 0 \right) = 1
    $$
    by countability. Thus, the set
    $$
    \mathcal A := \left\{\omega \in \tilde \Omega\,:\, \left(\tilde \Lambda^k(\omega)\,, \tilde N^k(\omega)  \right) \in \mathcal V^{n_k}\,,\,\,k\geq0 \right\}\,\bigcap \,\left\{\tilde \Lambda\,\,\text{continuous} \right\}\,,
    $$
    has unit measure, and for any $\omega \in \mathcal A\,$, we can apply Lemma \ref{lemma:stability_equation} with $\left(f_k\right)_{k \geq 0} := \left(\tilde N^k(\omega) \right)_{k \geq 0}$ and $\left(F_k\right)_{k \geq 0} := \left(\tilde \Lambda^k(\omega) - G_0^{n_k}\right)_{k \geq 0}\,$. This leads to $\tilde{P}\left(\left(\tilde \Lambda\,, \tilde N\right) \in \mathcal V \right) = 1\,$, and therefore,
    $$
    \Lambda_t = G_0(t) + \int_0^t \,K(t-s)\,N_s\,ds\,,\quad t \leq T\,, \quad a.s.
    $$
    We conclude by using the fact that $N$ has finite variation on $[0\,,T]$ since it is non-decreasing, which allows us to rewrite the previous equation as 
    $$
    \Lambda_t = \int_0^t\,\left(g_0(s) + \int_0^s\,K(s-r)\,dN_r \right)\,ds\,, \quad t \leq T\,, \quad a.s.\,,
    $$
    identifying $\lambda$ explicitly.

    Finally, for \textit{(iv)}, we have the bound 
    $$
    \sup_{n \geq 1}\, \sup_{0 \leq t \leq T}\, \mathbb E \left[\left(N^n_t - \Lambda^n_t\right)^2 \right] \leq C_Z
    $$
    from \eqref{eq:bound_exp_Z}, which shows that $\left(N^n_t - \Lambda^n_t\right)_{t \leq T, n \geq 1}$ is uniformly integrable. For each $n \geq 1\,$, $N^n - \Lambda^n$ is a martingale with respect to the filtration $\left(\mathcal F^n\right)_{t \leq T}$ generated by $N^n$ (see \eqref{eq:def_filtration}), which is also the filtration generated by $\left(\Lambda^n\,, N^n - \Lambda^n\right)$. Thus, applying \citet[Theorem~5.3]{whitt2007proofs} shows that $N- \Lambda$ is a martingale with respect to the filtration generated by $\left(\Lambda\,, N - \Lambda\right)\,$. This filtration is identical to the one generated by $\left(\Lambda\,, N\right)\,$, which is also the one generated by $N$ alone since we have shown that $\Lambda = G_0 + K * N$ almost surely. This concludes the proof.
\end{proof}

\subsection{Additional results}
We regroup here some technical results that are used in the previous subsections.

The following lemma is useful for proving the tightness of our processes.
\begin{lemma} \label{lemma:translation_continuous_L1}
Let $K \in L^1\left([0\,,T]\,, \mathbb R\right)\,$, we have 
$$
\lim_{\delta \to 0^+}\, \sup_{0 \leq \eta \leq \delta}\, \int_0^{T-\eta}\, \left|K(s + \eta) - K(s)\right|\,ds = 0\,.
$$
\end{lemma}
\begin{proof}
    Let $\varepsilon > 0\,.$ By density of continuous functions in $L^1\left([0\,,T]\,, \mathbb R\right)\,$, there is $f : [0\,,T] \to \mathbb R$ continuous such that $\|K - f\|_{L^1} \leq \varepsilon\,$. We can thus bound
    \begin{align*}
        \int_0^{T- \eta}\, \left|K(s + \eta) - K(s)\right|\, ds &\leq 2\,\|K-f\|_{L^1} + \int_0^{T-\eta}\, \left|f(s + \eta) - f(s)\right|\,ds \\
        &\leq 2\,\varepsilon + T\,w(f, \delta)\,, \quad \eta \leq \delta\,,
    \end{align*}
    which gives, by continuity of $f\,$,
    $$
    \lim_{\delta \to 0^+}\, \sup_{0 \leq \eta \leq \delta}\, \int_0^{T-\eta}\, \left|K(s + \eta) - K(s) \right|\, ds \leq 2 \, \varepsilon
    $$
    and concludes the proof.
\end{proof}

We also need to ensure some the stability of our Volterra equations in a deterministic setting.

\begin{lemma}
\label{lemma:stability_equation}
Let $\left(f_n\right)_{n \geq 1}$ be a sequence of càdlàg functions on $[0\,,T]$ converging to $f\,$, càdlàg with bounded variation, in $L^1([0\,,T])\,$. We define
$$
F_n(t) := \int_{[0,t]}\,f_n(t-s)\,K^n(ds)\,, \quad t \leq T\,, \quad n \geq 1\,,
$$
where the measures $\left(K^n\right)_{n \geq 1}$ are defined in Lemma \ref{lemma:Volterra}. If $F_n$ converges to $F$ continuous, in $L^1([0\,,T])\,$, then 
$$
F(t) = \int_0^t\,K(t-s)\,f(s)\,ds\,, \quad t \leq T\,.
$$
\end{lemma}
\begin{proof}
    For any $n \geq 1\,$, we can write 
    \begin{align*}
    \int_0^T\, \left|F(t) - \int_0^t\,K(t-s)\,f(s)\,ds\right|\,dt &\leq \int_0^T\,|F(t) - F_n(t)|\,dt + \int_0^T\,\left|\int_{[0,t]}\,\left(f_n(t-s) - f(t-s)\right)\,K^n(ds) \right|\,dt \\
    &\quad + \int_0^T\,\left|\int_{[0,t]}\,f(t-s)\,K^n(ds) - \int_0^t\,f(t-s)\,K(s)\,ds\right|\,dt \\
    &=: \mathcal I_1^n + \mathcal I^n_2 + \int_0^T\left|\mathcal I^n(t) - \mathcal I(t) \right|\,dt \,, \quad n \geq 1\,.
    \end{align*}
    By $L^1$-convergence, $\mathcal I^n_1$ goes to $0$ as $n$ goes to infinity. Similarly, 
    $$
    \mathcal I_2^n \leq K^n([0\,,T])\,\int_0^T\,|f_n(t) - f(t)|\,dt = \int_0^T\,K(s)\,ds\, \int_0^T\,|f_n(t) - f(t)|\,dt\,, \quad n \geq 1\,,
    $$
    which vanishes in the limit.

    For the remaining term, we use $K^n([0\,,T]) = \int_0^T\,K(s)\,ds\,$, and for $0 \leq t < T\,$,
    $$
    K^n([0,t]) = \sum_{j = 0}^{\lfloor nt / T\rfloor}\,\int_{t^n_j}^{t^n_{j+1}}\,K(s)\,ds = \int_0^T\, \1_{\left[\frac{T}{n},\left(\lfloor \frac{nt}{T}\rfloor + 1\right)\frac{T}{n}\right]}(s)\,K(s)\,ds \overset{n \to \infty}{\longrightarrow} \int_0^t\,K(s)\,ds\,,
    $$
    by the dominated convergence theorem since $K \in L^1([0\,,T])\,$. We then rewrite $\mathcal I^n(t)$ as
    \begin{align*}
    \mathcal I^n(t) &= \int_{[0,t]}\,f(t-s)\,K^n(ds) \\
    &= f(0)\,K^n([0\,,t]) + \int_{[0,t]}\,\int_0^{t-s}\,df(u)\,K^n(ds) \\
    &= f(0)\,K^n([0\,,t]) + \int_0^t\,K^n([0\,,t-u])\,df(u) \,, \quad n \geq 0\,,
    \end{align*}
    which, thanks to the domination $K^n([0\,,t]) \leq K^n([0\,,T]) = \int_0^T \,K(s)\,ds\,$, converges to 
    $$
    f(0)\,\int_0^t\,K(s)\,ds + \int_0^t\,\int_0^{t-u}\,K(s)\,ds\,df(u) = \int_0^t\,f(t-s)\,K(s)\,ds = \mathcal I(t)\,.
    $$
    Finally, we can apply the dominated convergence theorem with the bound 
    $$
    \left|\mathcal I^n(t) - \mathcal I(t)\right| \leq 2\,\left\|f\right\|_{\infty}\,\int_0^T\,K(s)\,ds\,, \quad t \leq T\,, \quad n \geq 1\,,
    $$
    to show that $\int_0^T\,\left|\mathcal I^n(t) - \mathcal I(t)\right|\,dt$ goes to $0$ as $n \to \infty\,$. We conclude that
    $$
    \int_0^T\,\left|F(t) - \int_0^t\,K(t-s)\,f(s)\,ds \right|\,dt = 0\,,
    $$
    so that, by continuity, 
    $$
    F(t) = \int_0^t \,K(t-s)\,f(s)\,ds\,, \quad t \leq T\,.
    $$
\end{proof}

Finally, the continuity of the following functional is used to prove that the limit of $(N^n)_{n \geq 1}$ is a counting process.

\begin{lemma}
\label{lemma:functionals_continuity}
On $D([0\,,T])\,$, the function 
$$
\mathcal S_{\Delta} : f \mapsto \sup_{0 \leq t \leq T}\, |\Delta f(t)|\,, \quad f \in D([0\,,T])\,,
$$
is continuous for the Skorokhod topology.
\end{lemma}
\begin{proof}
    Let $f_n \to f$ in the Skorokhod topology. According to \citet[Theorem~VI.1.14]{jacod2013limit}, there exists continuous increasing bijections $\left(\lambda_n\right)_{n \geq 0}$ from $[0\,,T]$ onto itself such that 
    $$
    \left\| f_n \circ \lambda_n - f\right\|_{\infty} \to 0\,, \quad \left\| \lambda_n - Id \right\|_{\infty} \to 0\,,
    $$
    as $n$ goes to infinity. Since
    \begin{align*}
        \left|\,|f_n(\lambda_n(t)) - f_n(\lambda_n(s))| - |f(t) - f(s) | \,\right| &\leq |f_n(\lambda_n(t)) - f(t) + f(s) - f_n(\lambda_n(s))| \\
        &\leq 2\,\|f_n \circ \lambda_n - f\|_{\infty}\,, \quad s \leq t \leq T\,, \quad n \geq 0\,,
    \end{align*}
    we deduce that 
    $$
    \left|\mathcal S_{\Delta}(f_n \circ \lambda_n) - \mathcal S_{\Delta}(f) \right| \leq 2\,\left\|f_n \circ \lambda_n - f\right\|_{\infty}\,, \quad n \geq 0\,,
    $$
    showing that $\mathcal S_{\Delta} (f_n \circ \lambda_n) \overset{n \to \infty}{\longrightarrow} \mathcal S_{\Delta}(f)\,$. We conclude by noticing that since for each $n \geq 0\,$ $\lambda_n$ is a continuous bijection from $[0\,,T]$ onto itself, $\mathcal S_{\Delta} (f_n \circ \lambda_n) = \mathcal S_{\Delta}(f_n)\,$.
\end{proof}

\begin{appendix}
\section{Kernels and resolvents}\label{app:kernels}
Consider a kernel $K \in L^1_{\mathrm{loc}}(\mathbb{R}_+, \mathbb{R}_+)$.
There exists a unique kernel $R \in L^1_{\mathrm{loc}}(\mathbb{R}_+, \mathbb{R}_+)$, called the \emph{resolvent of the second kind}, or simply the \emph{resolvent}, such that
\[
R * K = K * R = R - K.
\]
Existence and uniqueness are guaranteed by \citet*[Theorems 2.3.1 and 2.3.5]{Gripenberg1990}.

Resolvents are particularly useful for solving linear Volterra equations. More precisely, the solution to
\[
f = g + K * f
\]
is given by
\[
f = g + R * g.
\]

Below, we provide a table listing some frequently used kernels, along with their integrals $\bar K(t) := \int_0^t\,K(s)\,ds$ and the corresponding resolvents. We use $c \geq 0\,$, $b \in \R\,$, $\alpha > 0\,$, and we denote by $E_{\alpha, \beta}$ the Mittag--Leffler function defined by
$$
E_{\alpha, \beta}(z) := \sum_{n = 0}^{+ \infty}\,\frac{z^n}{\Gamma(\alpha\,n + \beta)}\,, \quad z \in \mathbb C\,, \quad \alpha\,, \beta > 0\,,
$$
and by $\gamma$ the lower incomplete gamma function 
$$
\gamma(\alpha, x) := \dfrac{1}{\Gamma(\alpha)}\int_0^xt^{\alpha - 1}e^{-t}\,dt\,, \quad x \geq 0\,, \quad \alpha > 0\,.
$$

\begin{equation*}
\renewcommand{\arraystretch}{2}
    \begin{array}{|c|c|c|}
        \hline
        K(t) & \bar K(t) & R(t)   \\
        \hline
        c\,e^{-bt} & \frac{c}{b}\,\left(1 - e^{-bt}\right) & c\,e^{-(b-c)t} \\
        \hline
        \dfrac{c\,t^{\alpha - 1}}{\Gamma(\alpha)} & \dfrac{c\,t^{\alpha}}{\Gamma(\alpha+1)} & c\,\,t^{\alpha - 1}\,E_{\alpha,\alpha}(c\,t^{\alpha}) \\
        \hline
        c e^{-bt}\dfrac{t^{\alpha-1}}{\Gamma(\alpha)} & \dfrac{c}{b^\alpha}\gamma(\alpha, bt) & c e^{-bt} t^{\alpha - 1}E_{\alpha, \alpha}(ct^\alpha)\\
        \hline
        c t^{\alpha - 1}E_{\alpha, \alpha}(ct^\alpha) &  c t^{\alpha}E_{\alpha, \alpha + 1}(ct^\alpha) & c  t^{\alpha - 1}E_{\alpha, \alpha}(2ct^\alpha) \\
        \hline
         c e^{-bt} t^{\alpha - 1}E_{\alpha, \alpha}(ct^\alpha) &  \sum_{n = 1}^{+ \infty}\,\left(\frac{c}{b^\alpha}\right)^n\gamma(\alpha n, bt) & c e^{-bt} t^{\alpha - 1}E_{\alpha, \alpha}(2ct^\alpha) \\
        \hline
    \end{array}
\end{equation*}
    
    \section{The Inverse Gaussian distribution}
    \label{App:IG}
    The Inverse Gaussian distribution, is a probability distribution on $\R_+$ given by the following density
    \begin{equation}
        \label{eq:IG_density}
        f_{\mu,\lambda}(x) = \sqrt{\frac{\lambda}{2\pi x^3}}\,\exp\left(- \frac{\lambda\,(x - \mu)^2}{2\mu^2 x} \right)\,,\quad x > 0\,, \quad \mu\,,\lambda > 0\,,
    \end{equation}
    If $X \sim IG(\mu\,,\lambda)\,$, it admits finite moments of any order, that can be computed using the recursive formula
    \begin{equation}
        \label{eq:IG_moments}
        \mathbb E \left[X\right] = \mu\,, \quad \text{and}\quad \mathbb E \left[X^{n}\right] = (2n - 3)\,\frac{\mu^2}{\lambda}\,\mathbb E \left[X^{n-1}\right] + \mu^2\,\mathbb E \left[X^{n-2}\right]\,, \quad n \geq 2 \,. 
    \end{equation}
    Finally, its characteristic function is given by 
    \begin{equation}
        \label{eq:IG_charac}
        \mathbb E \left[\exp\left( w\,X\right) \right] = \exp \left(\frac{\lambda}{\mu}\,\left(1 - \sqrt{1 - 2\,\frac{\mu^2}{\lambda}\,w} \right) \right)\,, \quad w \in \mathbb C\,, \quad \Re(w) \leq 0\,.
    \end{equation}

    Inverse Gaussian random variables can be simulated with one standard normal and one uniform variable, using the acceptance-rejection method from \citet*{michael1976generating}, as follows.
    
    \begin{algorithm}[H]
\caption{Sampling from the Inverse Gaussian Distribution}\label{alg:IG_sampling}
\begin{algorithmic}[1]
\State \textbf{Input:} Parameters \( \mu > 0 \), \( \lambda > 0 \).
\State \textbf{Output:} Sample \( IG \) from the Inverse Gaussian distribution.

\State   
Generate \( \xi \sim \mathcal{N}(0, 1) \) and compute \( Y = \xi^2 \).

\State Compute the candidate value \( X \):  
$$
X = \mu + \frac{\mu^2 Y}{2\lambda} - \frac{\mu}{2\lambda} \sqrt{4\mu\lambda Y + \mu^2 Y^2}.
$$

\State Generate a uniform random variable:  
Sample \( \eta \sim \text{Uniform}(0, 1) \).

\State Select the output:  
\If{$ \eta \leq \frac{\mu}{\mu + X} $}
    \State Set the output \( IG = X \).
\Else
    \State Set \( IG = \frac{\mu^2}{X} \).
\EndIf
\end{algorithmic}
\end{algorithm}
\end{appendix}

\bibliographystyle{plainnat}
\bibliography{main}

\end{document}